\documentclass[red,11pt,a4paper]{article}

\usepackage{xcolor}
\usepackage{amsmath}
\usepackage{amscd}
\usepackage{amssymb,amsthm}

\providecommand\mathbb{\bf}
\newcommand\R{{\mathbb R}}
\newcommand\N{{\mathbb N}}

\newtheorem{pro}{Proposition}[section]
\newtheorem{defi}{Definition}[section]

\newtheorem{lemma}{Lemma}[section]
\newtheorem{theorem}{Theorem}[section]
\newtheorem{remark}{Remark}[section]

\newcommand{\calo}{
{\mathcal O}}

\newcommand{\intth}[1]{
\int _0 ^{\pi} #1 \;\mathrm{d}\theta}

\newcommand{\Wp}{
W^\prime}

\newcommand{\omegap}[0]{
\omega ^\prime}

\newcommand{\mlo}[0]{
M_{l\Omega}}

\newcommand{\calmlo}[0]{
{\mathcal M}_{l\Omega}}

\newcommand{\tvarphi}[0]{
\widetilde{\varphi}}

\newcommand{\tpsi}[0]{
\widetilde{\psi}}

\newcommand{\sphere}[0]{
\mathbb{S}^{d-1}}

\newcommand{\rsphere}[0]{
r \mathbb{S}^{d-1}}

\newcommand{\supp}[0]{
\mathrm{supp\;}}

\newcommand{\abv}[0]{
(\alpha - \beta |v|^2) v}

\newcommand{\abvs}[0]{
(\alpha - \beta \left |{\mathcal V}(s;v)\right |^2) {\mathcal V}(s;v)}

\newcommand{\intxv}[1]{
\int _{\R^d}\!\!\int_ {\R^d}\! #1\;\mathrm{d}v \mathrm{d}x}

\newcommand{\intxpvp}[1]{
\int _{\R^d}\!\!\int_ {\R^d}\! #1\;\mathrm{d}v^\prime \mathrm{d}x^\prime}

\newcommand{\into}[1]{
\int _{r \mathbb{S}^{d-1}}\! #1\;\mathrm{d}\omega}

\newcommand{\intopr}[1]{
\int _{r \mathbb{S}^{d-1}}\! #1\;\mathrm{d}\omega ^{\;\prime}}

\newcommand{\eps}[0]{
\varepsilon}

\newcommand{\fe}[0]{
f ^\varepsilon}

\newcommand{\go}[0]{
g^{(1)}}

\newcommand{\fin}[0]{
f ^{\mathrm{in}}}

\newcommand{\fo}[0]{
f ^{(1)}}

\newcommand{\Fo}[0]{
F ^{(1)}}

\newcommand{\Go}[0]{
G ^{(1)}}

\newcommand{\ftw}[0]{
f ^{(2)}}

\newcommand{\A}[0]{
\{0\} \cup \rsphere}

\newcommand{\xA}[0]{
\R ^d \times ( \{0\} \cup  \rsphere )}

\newcommand{\Divx}[0]{
\mathrm{div}_x}

\newcommand{\Divv}[0]{
\mathrm{div}_v}

\newcommand{\Divo}[0]{
\mathrm{div}_\omega}

\newcommand{\ave}[1]{
\left \langle #1 \right \rangle }

\newcommand{\mbxv}[0]{
{\cal M}_b ^+ (\R ^d \times \R ^d)}

\newcommand{\mbv}[0]{
{\cal M}_b ^+ (\R ^d)}

\newcommand{\imoo}[0]{
\left ( I_d - \frac{\omega \otimes \omega}{r^2} \right ) }

\newcommand{\txi}[0]{
\widetilde{\xi}}

\newcommand{\tf}[0]{
\tilde{f}}

\newcommand{\czcxv}[0]{
C^0_c (\R^d \times \R^d)}

\newcommand{\czc}[0]{
C^0_c (\R^d)}

\newcommand{\coc}[0]{
C^1_c (\R^d)}

\newcommand{\vsv}[0]{
\frac{v}{|v|}}

\newcommand{\lime}[0]{
\lim _{\varepsilon \searrow 0}}

\newcommand{\dom}[0]{
\mathrm{d}\omega}

\newcommand{\dv}[0]{
\mathrm{d}v}

\newcommand{\dx}[0]{
\mathrm{d}x}

\newcommand{\intv}[1]{
\int _{\R ^d} \!#1 \;\mathrm{d}v}

\newcommand{\intvp}[1]{
\int _{\R ^d} \!#1 \;\mathrm{d}v^\prime}

\newcommand{\ind}[1]{
{\bf 1}_{#1}}

\newcommand{\vp}[0]{
v^{\prime}}

\baselinestretch\renewcommand{\baselinestretch}{1.5}

\numberwithin{equation}{section}

\begin{document}

\title{Reduced fluid models for self-propelled particles \\interacting through alignment}
\author{M. Bostan$^{1}$ and J. A. Carrillo$^2$}

\maketitle

\centerline{1. Aix Marseille Universit\'e, CNRS, Centrale Marseille, Institut de Math\'ematiques de Marseille, UMR 7373,}
\centerline{Ch\^ateau Gombert 39 rue F. Joliot Curie, 13453 Marseille Cedex 13 FRANCE}
 \bigskip
\centerline{2. Department of Mathematics, Imperial College London, }
\centerline{London SW7 2AZ, United Kingdom.}

\begin{abstract}
The asymptotic analysis of kinetic models describing the behavior
of particles interacting through alignment is performed. We will
analyze the asymptotic regime corresponding to large alignment
frequency where the alignment effects are dominated by the self
propulsion and friction forces. The former hypothesis leads to a
macroscopic fluid model due to the fast averaging in velocity,
while the second one imposes a fixed speed in the limit, and thus
a reduction of the dynamics to a sphere in the velocity space. The
analysis relies on averaging techniques successfully used in the
magnetic confinement of charged particles. The limiting particle
distribution is supported on a sphere, and therefore we are forced
to work with measures in velocity. As for the Euler-type
equations, the fluid model comes by integrating the kinetic
equation against the collision invariants and its generalizations
in the velocity space. The main difficulty is their identification
for the averaged alignment kernel in our functional setting of
measures in velocity.
\end{abstract}

\section{Introduction}
\label{Intro}

The subject matter of this paper concerns the behavior of living
organisms such as flocks of birds, school of fish, swarms of
insects, myxobacteria ... These models include  short-range
repulsion, long-range attraction, self-propelling and friction
forces, reorientation or alignment see
\cite{Aoki,Reyn,HW,VicCziBenCohSho95,LevRapCoh00,GreCha04,CouKraFraLev05,HCH09,BTTYB,BEBSVPSS}.
We consider self-propelled particles with Rayleigh friction
\cite{ChuHuaDorBer07,ChuDorMarBerCha07,CarDorPan09,UAB25,BKSUV,ABCV,CHM1,CHM2},
and alignment, introduced through the Cucker-Smale reorientation
procedure \cite{CS0,CS2}, see also
\cite{HT08,HL08,CFRT10,review,MT11,MT2} for further details and
\cite{KCBFL} for a survey. If we denote by $f = f(t,x,v)\geq 0$
the particle density in the phase space $(x,v) \in \R^d \times \R
^d$, with $d \in \{2,3\}$, the self-propulsion/friction mechanism
is given by the term $\Divv\{f \abv\}$. Notice that the balance
between the self-propulsion and friction forces occurs on the
velocity sphere $|v| = r := \sqrt{\alpha/\beta}$. We fix the speed
$r$, meaning that $\alpha$ and $\beta$ are anytime related by the
equality $\alpha = \beta r ^2$. The coefficients $\alpha, \beta
>0$ can be interpreted as follows. In the absence of friction, the
particles accelerate with $\alpha v$, leading to a exponential
growth of velocity, with frequency $\alpha$. In the absence of
self-propulsion, the inverse of the relative kinetic energy grows
linearly, with the frequency $2\beta |v|^2$, where $v$ is the
initial velocity of the particle
\[
\frac{\mathrm{d}}{\mathrm{d}s}\frac{|v|^2}{|V(s)|^2} = - \frac{|v|^2}{|V(s)|^4} 2 (V(s) \cdot V^\prime (s)) = 2 \beta |v|^2.
\]
Each individual in the group relaxes its velocity toward the mean
velocity of the neighbors, leading to the term $\nu \;\Divv\{f
(u[f] - v)\}$, where $\nu$ is the reorientation frequency and
$u[f]$ is the mean velocity
\[
u[f(t)](x) = \frac{\intxpvp{f(t, x^\prime, v^\prime) h(x- x^\prime) v^\prime}}{\intxpvp{f(t, x^\prime, v^\prime) h(x- x^\prime)}}.
\]
The weight application $h$ is a decreasing, radial, non negative
given function that determines the interaction neighborhood around
any position. By including also noise in the above kinetic model,
we get to the Fokker-Planck like equation
\begin{align}
\label{Equ1}
\partial _t f + \Divx(fv) + \Divv \{ f \abv\}  & = \nu \;\Divv \{f ( v - u[f])\} + \tau \Delta _v f  \\
& = \nu \;\Divv \{ f(v - u[f]) + \sigma \nabla _v f \}:=\nu
Q(f)\,, \nonumber
\end{align}
where $\sigma = \tau /\nu$ represents the diffusion coefficient in
the velocity space. We investigate the large time and space scale
regime of \eqref{Equ1} that is, we fix large time and space units.
In this case, equation \eqref{Equ1} should be replaced by
\begin{align}
\label{Equ2} \varepsilon_1 \{ \partial _t f + \Divx(fv)\} + \Divv \{f
\abv \} & = \nu Q(f).
\end{align}
The choice of a large length unit leads to a local reorientation
mechanism: the mean velocity $u[f]$ in \eqref{Equ2} is now given
by
\[
u[f(t)](x) = \frac{\intvp{f(t, x, v ^\prime)v ^\prime}}{\intvp{f(t,x,v^\prime)}}.
\]
Notice that if $f(t, x, \cdot) = 0$, then the Fokker-Planck
collision operator vanishes for any $u$. In this case we can
define $u[f(t)] = 0$, without loss of generality. We assume that
the frequencies $\varepsilon_1$ and $\nu$ scale like
$\tfrac{\nu}{\varepsilon_1} \approx \frac{1}{\eps _2}$ for some small
parameters $\eps _1, \eps _2 >0$ and thus the equation
\eqref{Equ2} becomes
\begin{align}
\label{Equ3}
\partial _t f ^{\eps_1, \eps _2} + \Divx(f ^{\eps_1, \eps _2}v ) & + \frac{1}{\eps_1} \Divv\{f ^{\eps_1, \eps _2}\abv\} = \frac{1}{\eps_2}Q(f ^{\eps_1, \eps _2}).
\end{align}
Assume for the moment that $\eps _1 \searrow 0$ and $\eps _2$ is
fixed. In this situation, the leading order term in the
Fokker-Planck equation \eqref{Equ3} corresponds to the
self-propulsion/friction mechanism, and we expect that the limit
density $f^{\eps_2} = \lim _{\eps _1 \searrow 0}f ^{\eps_1, \eps
_2}$ satisfies
\begin{equation*}
\Divv \{f^{\eps_2} \abv \} = 0.
\end{equation*}
The previous constraint exactly says that at any time $t$ and any
position $x$, the velocity distribution $f ^{\eps_2} (t, x,
\cdot)$ is a measure supported in $\{0\}\cup \rsphere$ cf.
\cite{BosCar13}. The particles will tend to move with asymptotic
speed $r$. These models have been shown to produce complicated
dynamics and patterns at the particle level such as mills, double
mills, flocks and clumps, see \cite{DorChuBerCha06}, whose
stability properties are very relevant in the applications, see
\cite{BKSUV,ABCV,CHM2}. Assuming that all individuals move with
constant speed also leads to spatial aggregation, patterns, and
collective motion \cite{CziStaVic97,EbeErd03,ParEde99}. More
exactly, it was shown in \cite{BosCar13} that, by taking the limit
$\eps _1 \searrow 0$, the solutions $f ^{\eps_1, \eps _2}$ of
\eqref{Equ3} converge toward the solution $f ^{\eps _2}$ of
\begin{align}
\label{Equ4}
\partial _t f ^{\eps _2} + \Divx(f ^{\eps _2} \omega ) + \frac1{\eps _2} \mathrm{div}_\omega \left \{ f ^{\eps _2}\imoo u [f ^{\eps _2}]\right \}& = \frac{\sigma}{\eps _2} \Delta _\omega f ^{\eps _2}
\end{align}
for all $(t,x,\omega) \in \R_+ \times \R^d \times \rsphere$ with
\[
u[f ^{\eps _2}(t)](x) = \frac{\into{f ^{\eps _2}(t,x,\omega)\omega}}{\into{f ^{\eps _2}(t,x,\omega)}},\;\;(t,x) \in \R_+ \times \R^d.
\]
The above result states that in the limit $\eps _1 \searrow 0$,
the Cucker-Smale model with diffusion is reduced to a Vicsek like
model, whose phase transition was analyzed in \cite{FL11}. The
evolution problem \eqref{Equ4} on the phase space $\R^d \times
\rsphere$, with normalized velocity field $u[f ^{\eps _2}]$ {\it
i.e.,}
\begin{equation*}
\partial _t f + \Divx(f\omega) + \nu \;\Divo \left \{ f \imoo \Omega [f]\right \} = \tau \Delta _\omega f,
\end{equation*}
for all $(t,x,\omega) \in \R_+ \times \R ^d \times \rsphere$ with
\[
\Omega [f(t)] (x)= \frac{\into{f(t,x,\omega) \omega }}{|\into{f(t,x,\omega) \omega}|},\;\;(t,x) \in \R_+ \times \R^d
\]
was also proposed in the literature as continuum version
\cite{DM08} of the Vicsek model
\cite{VicCziBenCohSho95,CouKraFraLev05}. Furthermore, the full
phase transition for stationary solutions and their asymptotic
stability was subsequently generalized in \cite{DFL10,DFL13}
allowing for quite general dependency of $\nu$ and $\tau$ on
$|u[f(t)]|$. We will focus on the relaxation toward the mean
velocity $u[f]$, whose alignment mechanism relies only on the
direction of the mean velocity $\Omega [f] = u[f]/|u[f]|$.
Nevertheless, our method still applies and allows us to handle the
model with normalization and the generalizations in
\cite{DM08,DFL13} as well.

The original kinetic Vicsek model in
\cite{VicCziBenCohSho95,CouKraFraLev02} was derived as the
mean-field limit of some stochastic particle systems in
\cite{BCC12}. In fact, previous particle systems have also been
studied with noise in \cite{BCC11} for the mean-field limit (see
also \cite{Neu77,BraHep77,Dob79,CCR10,AIR,CCH,CCH2,CCHS}), in
\cite{HLL09} for studying some properties of the Cucker-Smale
model with noise, and in \cite{BCCD16,CKR} for phase transitions
at the level of the Cucker-Smale model and the inhomogeneous level
respectively.

We assume now that both $\eps _1, \eps _2$ become small. The idea
is to justify a macroscopic model for \eqref{Equ4}, resulting from
the balance between two opposite phenomena
\begin{enumerate}
\item The reorientation, which tends to align the particle
velocities with respect to the mean velocity;

\item The diffusion, which tends to spread the particle velocities
isotropically on the sphere $\rsphere$.
\end{enumerate}
Such hydrodynamic models were obtained in \cite{DM08,DFL13}, by
letting $\eps _2 \searrow 0$ in the normalized alignment version
of \eqref{Equ4}. They are typically referred as Self-Organized
Hydrodynamics (SOH). Notice that the SOH model was obtained by
passing to the limit successively in \eqref{Equ3} with respect to
$\eps_1, \eps _2$. After letting $\eps _1 \searrow 0$, the
dynamics were reduced to the phase space $(x,v) \in \R^d \times
\rsphere$, but still captures microscopic behavior in the tangent
directions to the sphere $\rsphere$. The second limit procedure,
$\eps _2 \searrow 0$, leads to the macroscopic equations for the
density $\into{f}$ and the direction of the flux $\into{\omega
f}$.

We intend to obtain a SOH model, by passing to the limit in
\eqref{Equ3}, simultaneously with respect to $(\eps_1, \eps _2)$.
Motivated by the above discussion, we assume that $\eps _1 = \eps
^2$ and $\eps _2 = \eps$, where $\eps >0$ is a small parameter,
that is, the self-propulsion/friction mechanism dominates the
alignment. This implies that $\nu= \varepsilon$ and
$\tau=\sigma\varepsilon$. Therefore \eqref{Equ3} becomes
\begin{align}
\label{Equ6}
\partial _t \fe + \Divx(\fe v) + \frac1{\eps^2}\Divv \{ \fe \abv \}  = \frac1\epsilon
Q(f)\,,
\end{align}
for all $(t,x,v) \in \R_+ \times \R^{2d}$, supplemented by the
initial condition
\begin{equation*}
\label{Equ7} \fe (0,x,v) = \fin (x,v),\;\;(x,v) \in \R^d \times \R^d.
\end{equation*}
Very recently, by a similar scaling, fluid models have been
obtained for the transport of charged particles, under the action
of strong magnetic fields, which dominate the collision effects.
The resulting macroscopic model is a gyrokinetic version of the
Euler equations, in the parallel direction with respect to the
magnetic field \cite{BosDCDS15,BosIHP15}.

The behavior of the family $(\fe)_{\eps >0}$, as the parameter $\eps$ becomes small, follows by analyzing the formal expansion
\begin{equation}
\label{FormalExpansion} \fe = f + \eps \fo + \eps ^2 \ftw + ...
\end{equation}
Plugging the above Ansatz into \eqref{Equ6}, leads to the constraints
\begin{equation}
\label{Equ8}
\Divv\{f \abv \} = 0
\end{equation}
\begin{equation}
\label{Equ9} \Divv\{\fo  \abv \} = \Divv \{ f(v - u[f]) + \sigma
\nabla _v f\}
\end{equation}
and to the time evolution equations
\begin{align}
\label{Equ10}
\partial _t f + \Divx(fv) + \Divv\{\ftw \abv\} & = \mathcal{L}_f(f^{(1)})
\end{align}
with
$$
\mathcal{L}_f(f^{(1)}):= \Divv\{\fo (v - u[f]) + \sigma \nabla _v
\fo \} - \;\Divv \left \{f \;\frac{\intvp{\fo (\vp -
u[f])}}{\intvp{f}} \right \}
$$
cutting the development at second order.

We expect the same macroscopic SOH model for the moments of $f$ as
obtained in \cite{DM08,DFL10,DFL13}. The main advantage for
considering \eqref{Equ6} instead of \eqref{Equ4} with $\eps _2 =
\eps$ is that the resolution of \eqref{Equ6} for small $\eps$ will
provide a solution supported near $\R^d \times \rsphere$, which
fits much better the behavior of living organism systems, than the
solution of \eqref{Equ4} on $\R^d \times \rsphere$. But the price
to pay is to deal with two Lagrange multipliers, appearing in
\eqref{Equ10}, which have to be eliminated, thanks to the
constraints \eqref{Equ8} and \eqref{Equ9}. The first constraint
was analyzed in detail in \cite{BosCar13}. It exactly says that
$f$ is a measure supported in $\xA$. We denote by $\mbv{}$ the set
of non negative bounded Radon measure on $\R^d$.

\begin{pro}\label{FirstConstraint}
Assume that $(1+|v|^2)F \in {\cal M}_b ^+ (\R^d)$. Then $F$ solves $\Divv\{F \abv \} = 0$ in ${\mathcal D}^\prime (\R^d)$ {\it i.e.,}
\[
\int_{\R^d} \abv \cdot \nabla _v \varphi \;\mathrm{d} F(v) = 0,\;\;\mbox{for any } \varphi \in \coc{}
\] if and only if $\supp F \subset \A$.
\end{pro}
The proof of Proposition \ref{FirstConstraint} is based on the resolution of the adjoint problem
\[
-(\alpha - \beta |v|^2) v \cdot \nabla _v \varphi = \psi (v),\;\;v \in \R^d
\]
for any smooth function $\psi$ with compact support in
$\R^d \setminus (\A)$, cf. Lemma 3.1 of \cite{BosCar13}.

\begin{lemma} \label{AdjointProblem}
For any $C^1$ function $\psi = \psi (v)$ with compact support in
$\R^d \setminus (\A)$, there is a bounded $C^1$ function $\varphi =
\varphi (v)$ such that $\varphi (0) = 0$ and
\begin{equation*}
- \abv \cdot \nabla _v \varphi = \psi (v),\;\;v \in
\R^d.
\end{equation*}
\end{lemma}
In the sequel, we introduce a projection operator onto the
subspace of the constraints in \eqref{Equ8}. This construction
follows closely the gyro-average method in gyro-kinetic theory
\cite{BosAsyAna,BosTraEquSin,BosGuiCen3D,Bos12,BosCalQAM1,BosCalQAM2,BosAnisoDiff}.
An average operator serves to separate between two scales. For
example, in gyro-kinetic theory, two time scales exist: a fast
time variable, related to the rapid cyclotronic motion, and a slow
time variable, related to the parallel motion with respect to the
magnetic field. The gyro-average operator represents the average
of the fast dynamics over a cyclotronic period, provided that the
slow time variable is frozen. Following this technique, we obtain
an accurate enough but simpler model, from the numerical
approximation point of view. All the fluctuations have been
removed and replaced by averaged effects.

Our model \eqref{Equ6} presents not two, but three time variables:
$t, t/\eps$ and $t /\eps ^2$. The dynamics are dominated by the
self-propulsion/friction mechanism, introducing the fast time
variable $s = t/\eps ^2$. The average operator is related to the
characteristic flow of the field $\frac1{\eps ^2} \abv \cdot
\nabla _v$. This characteristic flow ${\mathcal V} = {\mathcal V}
(s;v)$, written with respect to $s = t/\eps ^2$
\[
\frac{\mathrm{d}{\mathcal V} }{\mathrm{d}s} = \abvs,\;\;{\mathcal V} (0;v) = v
\]
conserves the direction $\vsv$ and has as equilibria the elements
of $\A$. The Jacobian matrix is given by
\[
\partial _v \{\abv \} = (\alpha - \beta |v|^2) I_d - 2 \beta v\otimes
v\,.
\]
Being negative on $\rsphere$ and definite positive at $0$, we
deduce that the points of $\rsphere$ are stable equilibria, and
$0$ is an unstable equilibrium. For simplicity, we neglect the
measure of the unstable point $0$ in the velocity space and assume
that this is not present in the limit $\epsilon\to 0$ at any level
of the expansion. As we elaborate below, we will rigorously
compute the terms in the expansion needed to derive formally the
hydrodynamic equations. The complete mathematical analysis of the
limiting procedure is out of scope of this paper. We are mainly
interested in the two or three dimensional setting, but the same
arguments apply for any dimension $d\geq 2$. For the sake of
generality, we state and prove all the results in any dimension
$d\geq 2$, and we distinguish, if necessary, between the cases $d
= 2$ and $d\geq 3$.

Motivated by the previous observations, we define the average of a
non negative bounded measure cf. \cite{BosCar13}. We will denote
by $f(x,v)\;\mathrm{d}v \mathrm{d}x$ the integration against the
measure $f$. This is done independently of being the measure $f$
absolutely continuous with respect to the Lebesgue measure or not.

\begin{defi}$\;$\\
1. Let $F \in {\mathcal M}_b ^+ (\R^d)$ be a non negative bounded
measure on $\R^d$. We denote by $\ave{F}$ the measure
corresponding to the linear application
\[
\psi \to \intv{\psi(v)\,\ind{v = 0}  F(v) } +
\intv{\psi\left ( r \vsv \right ) \ind{v \neq 0} F(v)}\,,
\]
for all $\psi \in \czc$, {\it i.e.,}
\[
\intv{\psi(v) \ave{F}(v)} = \int _{v = 0} \psi(v)F(v)\,\dv + \int _{v \neq 0} \psi \left ( r \vsv \right )F(v)\,\dv\,,
\]
for all $\psi \in \czc$.\\
2.
Let $f \in \mbxv{}$ be a non negative bounded measure on $\R^d \times \R^d$. We denote by $\ave{f}$ the measure corresponding to the linear application
\[
\psi \to \intxv{\psi(x,v)\,\ind{v = 0}  f(x,v) } +
\intxv{\psi\left ( x , r \vsv \right ) \ind{v \neq 0} f(x,v)}\,,
\]
for all $\psi \in \czcxv$, {\it i.e.,}
\[
\intxv{\psi(x,v) \ave{f}(x,v)} = \int _{v = 0} \!\!\!\!\!\!\psi(x,v)
f(x,v)\,\dv \dx + \int _{v \neq 0} \!\!\!\!\!\!\psi \left ( x , r \vsv \right )
f(x,v)\,\dv \dx,
\]
for all $\psi \in \czcxv$.
\end{defi}
It is easily seen that the average of a non negative bounded
measure is a non negative bounded measure, with the same mass, but
supported in $\A$, $\R^d \times (\A)$ respectively. We have the
following characterization (see Proposition 5.1 \cite{BosCar13}).
\begin{pro} \label{VarChar}
Assume that $f$ is a non negative bounded measure on $\R^d \times
\R^d$. Then $\ave{f}$ is the unique measure $\tf$ satisfying $\supp
\tf \subset \R^d \times (\A)$,
\[
\int _{v\neq 0} \psi \left ( x , r \vsv \right )\tf(x,v)\,\dv \dx =
\int _{v \neq 0}\psi \left ( x , r \vsv \right
)f(x,v)\,\dv \dx,\;\;\psi \in \czcxv{}
\]
and $\tf = f$ on $\R^d \times \{0\}$.
\end{pro}
A direct consequence of Proposition \ref{VarChar} is that any
bounded, non negative measure, supported in $\R^d \times (\A)$ is
left unchanged by the average operator. Another property of the
average operator is that it removes any measure of the form $\Divv
\{f \abv \}$, cf. Proposition 5.2 \cite{BosCar13}. 
\begin{pro}
\label{Elimination}
For any $f \in \mbxv{}$ such that $ \Divv \{f \abv
\}\in {\cal M}_b (\R^d \times \R^d)$, we have $\ave{\Divv \{f \abv
\}} = 0$.
\end{pro}
The above proposition plays a crucial role when eliminating the
Lagrange multiplier $\ftw$ in \eqref{Equ10}. Indeed, for doing
that, it is enough to average both hand sides in \eqref{Equ10}. By
the constraint \eqref{Equ8}, we know that $f$ is supported in
$\R^d \times (\A)$, and thus is left invariant by the average. We
check that $\ave{\partial _t f } = \partial _t \ave{f} = \partial
_t f$, and thus, averaging \eqref{Equ10} still leads to a
evolution problem for $f$
\begin{align}
\label{Equ12}
\partial _t f + \ave{\Divx(fv)} & =
\ave{ \mathcal{L}_f(f^{(1)}) }.
\end{align}
Certainly, a much more difficult task is to eliminate the Lagrange
multiplier $\fo$. We expect that this can be done thanks to the
constraint in \eqref{Equ9}. The solvability of \eqref{Equ9}, with
respect to $\fo$, depends on a compatibility condition, to be
satisfied by the right hand side. Indeed, by Proposition
\ref{Elimination}, we should have
\begin{equation}
\label{Equ14} \ave{\Divv\{f (v - u[f]) + \sigma \nabla _v f \}} =
\ave{\Divv \{ \fo \abv\}}=0
\end{equation}
saying that $f$ is a equilibrium for the average collision kernel
$\ave{Q(f)} = 0$. The equilibria of the average collision kernel
form a $d-1$-dimensional manifold, that is one dimension less than
the equilibria manifold of the Fokker-Planck operator $Q$ (see
also \cite{DM08,FL11}). For any $l \in \R_+, \Omega \in \sphere$,
we introduce the von Mises-Fisher distribution
\[
\mlo (\omega)  \;\dom= \frac{\exp \left ( l \Omega \cdot \frac{\omega}{r} \right )}{\intopr{\exp \left ( l \Omega \cdot \frac{\omega ^\prime}{r} \right )}}\;\dom,\;\;\omega \in \rsphere.
\]
\begin{pro}
\label{AveEqui}
Let $F \in {\mathcal M}_b ^+(\R^d)$ be a non negative bounded measure on $\R^d$, supported in $\rsphere$. The following statements are equivalent:\\
1. $\ave{Q(F)} = 0$, that is
\[
\int _{v \neq 0} \left \{-(v- u[F])\cdot \nabla _v \left [ \tpsi
\left ( r \vsv\right ) \right ] + \sigma \Delta _v \left [ \tpsi
\left ( r \vsv\right ) \right ]   \right \} F\;\dv = 0,
\]
for all $\tpsi \in C^2(\rsphere)$.

\noindent 2. There are $\rho \in \R_+, \Omega \in \sphere$ such
that $F = \rho \mlo \dom$ where $l \in \R_+$ satisfies
\begin{equation}\label{kkk}
\frac{\int _0 ^\pi \cos \theta \;e ^{l \cos \theta} \sin ^{d-2}
\theta\;\mathrm{d}\theta}{\int _0 ^\pi e ^{l \cos \theta} \sin
^{d-2} \theta\;\mathrm{d}\theta} = \frac{\sigma }{r^2} l.
\end{equation}
\end{pro}

\vskip6pt

The modulus of the mean velocity is not a coordinate on the
equilibria manifold, but it is determined by the condition $ |u| =
\tfrac{\sigma l}{r}$ where $l$ satisfies \eqref{kkk}. Clearly $l =
0$ is a solution, which corresponds to the isotropic equilibrium
\[
F = \rho M_{0 \Omega} \;\dom = \rho \frac{\dom}{\bar\omega _d r
^{d-1}}
\]
where $\bar\omega _d$ represents the area of the unit sphere in
$\R^d$. The next proposition is essentially contained in
Proposition 3.3 in \cite{FL11}. We present a simplified proof,
based on computations with Bessel functions. 
\begin{pro}
\label{IsotropicEquilibrium}
Let $\lambda : \R_+ \to \R$ be the function given by
\[
\lambda (l) = \frac{\intth{\cos \theta e ^{l \cos \theta} \sin ^{d-2}\theta}}{\intth{e ^{l \cos \theta} \sin ^{d-2} \theta}},\;\;l \in \R_+,\;\;d \geq 2.
\]
The function $\lambda $ is strictly increasing, strictly concave and verifies
\[
\lambda (0) = 0,\;\;\lambda ^{\prime}(0) = \frac{1}{d},\;\;\lim_{l \to +\infty} \lambda (l) = 1.
\]
If $\frac{\sigma }{r^2} \geq \frac{1}{d}$, then the only solution of $\lambda (l) = \frac{\sigma }{r^2} l$ is $l = 0$. If $\frac{\sigma}{r^2} \in ]0,\frac{1}{d}[$, then there is a unique $l = l \left ( \frac{\sigma}{r^2}\right) >0$ such that $\lambda (l) = \frac{\sigma}{r^2}l$.
\end{pro}

In order to find the equations for the evolution of the density
$\rho$ and orientation $\Omega$, we need to find $f^{(1)}$ from
\eqref{Equ9} in order to feed the terms needed in \eqref{Equ12}.
However, we will see that this is not possible. We will need to
introduce a notion of generalized collision invariants, quite
related intuitively to the one introduced in
\cite{DM08,DFL10,DFL13}, in our functional setting of measures
supported in $\rsphere$ to avoid the computation of the full
$f^{(1)}$. This is the main technical difficulty due to the
measure functional setting since the precise definition of
generalized collision invariant we need is more involved than in
\cite{DM08,DFL10,DFL13}. Let us mention that this notion of
generalized collision invariant has been used in other related
models in collective dynamics \cite{DMY,DFM,DFMT} and in kinetic
models of wealth distribution \cite{DLR}.

Our main result establishes the macroscopic equations satisfied by
the density $\rho$ and orientation $\Omega$, which parameterize
the von Mises-Fisher equilibrium, obtained when passing to the
limit for $\eps \searrow 0$ in \eqref{Equ6}. We retrieve exactly
the limit SOH hydrodynamic model in \cite{DFL10}, written for any
space dimension $d\geq 2$ with the same explicit constants.
\begin{theorem}
\label{MainResult1}
For any $\sigma, r$ such that $\frac{\sigma}{r^2} \in ]0,\frac{1}{d}[$, we denote by $l = l \left ( \frac{\sigma}{r^2}\right)$ the unique positive solution of $\lambda (l) = \frac{\sigma}{r^2}l$. Let $\fin \in \mbxv$ be a non negative bounded measure on $\R^d \times \R^d, d\geq 2$. For any $\eps >0$ we consider the problem
\begin{align}
\label{Equ71}
\partial _t \fe + \Divx(\fe v) + \frac{1}{\eps^2} \Divv  ( \fe \abv ) & = \frac{1}{\eps} \Divv \{\fe (v - u[\fe]) + \sigma \nabla _v \fe\}
\end{align}
for all $(t,x,v) \in \R_+ \times \R ^d \times \R^d$ with $\fe (0)
= \fin$, $(x,v) \in \R^d \times \R^d$. Therefore the limit
distribution $f = \lime \fe$, is a von Mises-Fisher equilibrium $f
= \rho \mlo (\omega) \;\dom$ on $\rsphere$, where the density
$\rho (t,x)$ and the orientation $\Omega (t,x)$ satisfy the
macroscopic equations
\begin{equation}
\label{Equ73}
\partial _t \rho + \Divx \left (\rho  \frac{l\sigma}{r} \Omega\right ) = 0,\;\;(t,x) \in \R_+ \times \R^d
\end{equation}
\begin{equation}
\label{Equ74}
\partial _t \Omega + k_d \;r ( \Omega \cdot \nabla _x) \Omega + \frac{r}{l} (I_d - \Omega \otimes \Omega) \frac{\nabla _x \rho }{\rho} = 0
\end{equation}
with the initial conditions
\[
\rho(0,x) = \intv{\fin(x)},\;\;\Omega (0,x) = \frac{\intv{v \fin (x)}}{\left|\intv{v \fin (x)}\right|},\;\;x \in \R^d
\]
where
\[
k_d = \frac{\intth{e ^{l \cos \theta} \chi (\cos \theta) \cos \theta \sin ^{d-1} \theta}}{\intth{e ^{l \cos \theta} \chi (\cos \theta) \sin ^{d-1} \theta}}
\]
and $\chi$ solves
\begin{equation*}
-\frac{\sigma}{r^2} \frac{\mathrm{d}}{\mathrm{d}c} \left \{e^{lc}
\chi^{\;\prime} (c) (1-c^2)^{\frac{1}{2}}   \right \} =
re^{lc},\;\;c \in ]-1,1[,\;\;\chi (-1) = \chi (1) = 0\;\mbox{ if }
d = 2
\end{equation*}
and
\begin{align*}
-\frac{\sigma}{r^2} \frac{\mathrm{d}}{\mathrm{d}c} \left \{e^{lc} \chi^{\;\prime} (c) (1-c^2)^{\frac{d-1}{2}} \right \}
+ (d-2) \frac{\sigma}{r^2}e^{lc} \chi (c) (1-c^2)^{\frac{d-5}{2}}& = re^{lc}(1-c^2)^{\frac{d-2}{2}}\\
& \;c \in ]-1,1[,\; d \geq 3.
\end{align*}
\end{theorem}
A nice practical implication of our main result is that this
penalization procedure, by imposing asymptotically a cruise speed
for particles, could lead to efficient and stable numerical
schemes to compute the hydrodynamic equations
\eqref{Equ73}-\eqref{Equ74}. This is important due to the possible
non-hyperbolicity of the system \eqref{Equ73}-\eqref{Equ74}, see
\cite{DFL13}. The local in time well-posedness of the SOH system
\eqref{Equ73}-\eqref{Equ74} was studied in \cite{DLMP}. We finally
emphasize that the constants appearing in the equations
\eqref{Equ73}-\eqref{Equ74} coincide exactly with the ones
obtained in \cite{DFL13} after some easy but tedious algebraic
manipulations.

Our article is organized as follows. In Section \ref{AveColOpe} we
study the equilibria of the average collision operator in our
functional setting. This analysis can be carried out by
introducing some Bessel functions. In the next section we
investigate the notion of collision invariant suitable in our
functional setting. We determine the structure of these invariants
and present their symmetries. Section \ref{LimMod} is devoted to
the derivation of the fluid model for the macroscopic quantities,
parameterizing the limit von Mises-Fisher equilibrium. The proofs
of some technical results can be found in the Appendix.


\section{The equilibria of the average collision operator}
\label{AveColOpe}

We consider the collision operator $Q(F) = \Divv\{F(v - u[F]) +
\sigma \nabla _v F\}$ where $u[F] = \intv{v F}/\intv{F}$ is the
mean velocity. The above operator should be understood in the
duality sense between non negative bounded measures on $\R^d$ and
smooth functions, compactly supported in $\R^d$
\[
\intv{\psi (v) Q(F)} = \intv{[- (v - u[F]) \cdot \nabla _v \psi
(v) + \sigma \Delta _v \psi (v)] F }
\]
for any $F \in {\cal M}_b ^+ (\R^d)$ and $\psi \in C^2 _c (\R^d)$
such that $\intv{|v|F} < +\infty$. As suggested by the formal
expansion \eqref{FormalExpansion}, we focus on measures satisfying
(see \eqref{Equ8}-\eqref{Equ9})
\[
\Divv\{F \abv\} = 0,\;\;Q(F) = \Divv\{\Fo \abv  \}.
\]
Thanks to Propositions \ref{Elimination} and
\ref{FirstConstraint}, we deduce that $\supp F \subset \A$ and
\[
\ave{Q(F)} = \ave{\Divv\{\Fo \abv  \}} = 0.
\]
We discuss the case of non negative bounded measures supported on
the sphere $\rsphere$, that is, we discard all difficulties
related to the mass of the points at rest. For such measures, the
equality $\ave{Q(F)} = 0$ can be interpreted in the following
sense (see Proposition \ref{VarChar})
\[
\int _{v \neq 0} \left \{ - (v - u[F]) \cdot \nabla _v \left [ \tpsi \left ( r \vsv \right ) \right ] + \sigma \Delta _v \left [ \tpsi \left ( r \vsv \right ) \right ]\right \} = 0,\;\forall \tpsi \in C^2  (\rsphere).
\]
The complete description of the above equilibria of the average
collision operator $Q$, called the von Mises-Fisher distributions,
is given by Proposition \ref{AveEqui}, whose proof is detailed
below. We start with the following easy integration by parts
formula on spheres. The proof is postponed to \ref{A}.
\begin{lemma}
\label{IntByPartsSphere}
Assume that $A = A(v)$ is a $C^1$ vector field in ${\mathcal O} = \{v \in \R^d\;:\; r_1 < |v| < r_2 \}$. Then for any $t \in ]r_1, r_2[$ we have
\begin{equation}
\label{Ident0}
\int_{|\omega | = t} (\Divv A)(\omega) \;\dom = \int_{|\omega | = t}\left \{\frac{\omega \otimes \omega}{t^2} : \partial _v A (\omega) + \frac{(d-1)\omega}{t^2} \cdot A(\omega) \right \} \;\dom.
\end{equation}
In particular, if $A(v) \cdot v = 0, v \in {\mathcal O}$, then
\begin{equation}
\label{Ident1} \int_{|\omega | = t} (\Divv A) (\omega) \;\dom = 0,\;\;t \in ]r_1, r_2[
\end{equation}
and for any function $\chi \in C^1 ({\mathcal O})$ we have
\begin{equation}
\label{Ident2}
\int_{|\omega | = t} \nabla _v \chi (\omega) \cdot A(\omega) \;\dom + \int_{|\omega | = t} \chi (\omega) ( \Divv A) (\omega) \;\dom = 0,\;\;t \in ]r_1, r_2[.
\end{equation}
\end{lemma}
\vskip6pt
It is very convenient to express the differential operators
$\nabla _\omega, \Divo$ of functions and vector fields on the
sphere $\rsphere$ in terms of the differential operators $\nabla
_v, \Divv$ applied to extensions of functions and vector fields on
a neighborhood of $\rsphere$ in $\R^d$. The notation
$\widetilde{\cdot}$ stands for the restriction on the sphere
$\rsphere$ and $\tilde{\cdot}^t$ for the restriction on the sphere
$t\sphere$. The proof of the following lemma is detailed in
\ref{B}. 
\begin{lemma}
\label{Extension}$\;\\$
1. Let $\psi = \psi (v)$ be a $C^1$ function in a open set of $\R^d$, containing $\rsphere$. Then, for any $\omega \in \rsphere$ we have
\begin{equation*}
\label{Extens1} \nabla _\omega \tpsi (\omega) = \imoo \widetilde{\nabla _v \psi } (\omega).
\end{equation*}
2. Let $\tpsi = \tpsi (\omega)$ be a $C^1$ function on $\rsphere$ and $\psi : \calo = \{ v \in \R^d : r_1 < |v| < r_2 \} \to \R$ be the function defined by $\psi (v) = \tpsi \left ( r \vsv\right ), v \in \calo$, with $0 < r_1 < r < r_2 < +\infty$. Then, for any $t \in ]r_1, r_2 [$, we have
\begin{equation*}
\label{EquTangGrad} (\nabla _v \psi )(\omega _t) = ( \nabla
_{\omega _t} \widetilde\psi^t )(\omega _t) = \frac{r}{t} ( \nabla
_\omega \tpsi ) \left ( r \frac{\omega _t}{t} \right ),\;\;|\omega
_t | = t.
\end{equation*}
3. Let $\txi = \txi(\omega)$ be a $C^1$ tangent vector field on
$\rsphere$ and $\xi = \xi (v)$ a $C^1$ extension of $\txi$ in the
set $\calo = \{v \in \R^d : r_1 < |v|< r_2 \}$ such that $\xi (v)
\cdot v = 0$ for any $v \in \calo$. Then we have
\begin{equation*}
\label{Extens2}
(\Divo \txi )(\omega)= (\widetilde{\Divv \xi })(\omega),\;\;\omega \in \rsphere.
\end{equation*}
4. Let $\txi = \txi (\omega)$ a $C^1$ tangent vector field on
$\rsphere$ and $\xi (v) = \txi\left ( r \vsv \right ), v \in
\R^d\setminus \{0\}$, then
\begin{equation}\label{Complement}
(\mathrm{div}_{\omega _t } \xi )(\omega _t) = \frac{r}{t} (\Divo
\txi ) \left(\frac{r}{t} \omega _t \right),\;\;|\omega _t | = t.
\end{equation}
\end{lemma}

\vskip 6pt

Before giving the proof of Proposition \ref{AveEqui}, we indicate
a formula which will be used several times in our computations.
For any continuous function $G : [-r, r] \to \R$, $d \geq 2,
\Omega \in \sphere$, we have
\[
\into{G(\omega \cdot \Omega) } = \intth{G(r\cos \theta) \sin
^{d-2} \theta}\; r^{d-1} \bar\omega _{d-1}
\]
with $\omega _1 = 2$. In particular, for any continuous function $g : [-r, r] \to \R$, we have
\begin{align}
\label{ForMag}
\into{g(\omega \cdot \Omega) \mlo (\omega)} & = \frac{\into{g(\omega \cdot \Omega)\exp \left ( l \Omega \cdot \frac{\omega}{r}\right)}}{\into{\exp \left ( l \Omega \cdot \frac{\omega }{r} \right )}} \nonumber \\
& = \frac{\intth{g(r \cos \theta) e ^{l \cos \theta}\sin ^{d-2} \theta }}{\intth{e ^{l \cos \theta} \sin ^{d-2} \theta }}.
\end{align}
\begin{proof} (of Proposition \ref{AveEqui})\\
1.$\implies 2.$ We assume that $F$ is a equilibrium for the
average collision kernel. We claim that $\intv{\varphi (v) F} = 0$
for any continuous function $\varphi$ satisfying $\into{\varphi
(\omega) M(\omega)} = 0$, with $M(v) = \exp \left ( -\frac{|v -
u[F]\;|^2}{2\sigma} \right ), v \in \R^d$. The idea is to solve
the problem
\begin{equation}
\label{Equ14}
- \Divo (M(\omega) \nabla _\omega \tpsi ) = M(\omega) \tvarphi (\omega),\;\;\omega \in \rsphere
\end{equation}
where $\tvarphi$ is the restriction on $\rsphere$ of $\varphi$ as
usual. Notice that we have
\[
\into{\tvarphi (\omega) M(\omega) } = \into{\varphi (\omega) M(\omega) } = 0.
\]
We introduce the Hilbert spaces
\[
L^2 (\rsphere) = \{ \chi : \rsphere \to
\R,\;\;\into{\chi^2(\omega) M(\omega)} <+\infty\}
\]
\[
H^1 (\rsphere) = \{ \chi : \rsphere \to \R,\;\;\into{\{\chi^2 +
|\nabla _\omega \chi |^2 \}(\omega)M(\omega)} <+\infty\}
\]
endowed with the scalar products
\[
(\chi, \theta)_r = \into{\chi (\omega) \theta (\omega) M(\omega)},\;\;\chi, \theta \in L^2 (\rsphere)
\]
\[
((\chi, \theta))_r = \into{[\chi (\omega) \theta (\omega) + \nabla _\omega \chi \cdot \nabla _\omega \theta] \;M(\omega)},\;\;\chi, \theta \in H^1 (\rsphere).
\]
We denote by $|\cdot|_r, \|\cdot \|_r$ the norm induced by the above scalar products. There is a constant $C_r$ such that the following Poincar\'e inequality holds true
\[
|\chi |_r ^2 = \into{(\chi (\omega))^2 M(\omega)} \leq C_r \into{|\nabla _\omega \chi |^2 M(\omega)}= C_r |\nabla _\omega \chi |^2 _r
\]
for any $\chi \in H^1(\rsphere)$ satisfying $\into{\chi (\omega)
M(\omega) } = 0$. The previous inequality guarantees that the
application $\chi \to |\nabla _\omega \chi |_r$ is a norm
equivalent to $\|\cdot \|_r$ on
\[
\tilde{H}^1 (\rsphere) : = H^1 (\rsphere) \cap \{ \theta \in L^2 (\rsphere)\;:\; \into{\theta (\omega) M(\omega) } = 0\}.
\]
Therefore, the bilinear form
\[
(\chi, \theta) \in \tilde{H}^1 (\rsphere) \times \tilde{H}^1 (\rsphere) \to \into{\nabla _\omega \chi \cdot \nabla _\omega \theta \;M(\omega) }
\]
is symmetric, bounded and coercive. By the Lax-Milgram lemma,
there is a unique solution $\tpsi \in \tilde{H}^1 (\rsphere)$ for
the variational problem \eqref{Equ14} leading to
\begin{equation}
\label{Equ15}
\into{\nabla _\omega \tpsi \cdot \nabla _\omega \chi  \;M(\omega) } = \into{\tvarphi (\omega) \chi (\omega) M(\omega) }
\end{equation}
for any $\chi \in \tilde{H}^1 (\rsphere)$. Observe that
\eqref{Equ15} still holds true for any constant function on
$\rsphere$, thanks to the compatibility condition $\into{\tvarphi
(\omega) M(\omega)} = 0$. Therefore the variational formulation is
valid for any function $\chi \in H^1(\rsphere)$, implying that
\[
- \Divo (M(\omega) \nabla _\omega \tpsi ) = M(\omega) \tvarphi (\omega),\;\;\omega \in  \rsphere.
\]
We consider the extension of $\tpsi$ defined as usual as
\[
\psi (v) = \tpsi \left ( r \vsv \right ),\;\;v \in \R^d \setminus \{0\}.
\]
By Lemma \ref{Extension}, statements $2$ and $3$, we check that for any $v \in \rsphere$ we have
\begin{align*}
M(v) \left \{\frac{v - u[F]}{\sigma} \cdot \nabla _v \left [ \tpsi
\left ( r \vsv \right )\right ] - \Delta _v  \left [ \tpsi \left (
r \vsv \right )\right ] \right \} &= - \Divo (M \nabla _\omega
\tpsi ) \\
&= M(v) \tvarphi (v)
\end{align*}
and therefore we obtain
\[
\intv{\varphi (v) F} = \intv{ \left \{\frac{v - u[F]}{\sigma} \cdot \nabla _v \left [ \tpsi \left ( r \vsv \right )\right ] - \Delta _v  \left [ \tpsi \left ( r \vsv \right )\right ] \right \} F } = 0.
\]
We deduce that the linear forms $\varphi \to \into{\varphi
(\omega) M(\omega) }$ and $\varphi \to \intv{\varphi (v) F}$ are
proportional, see Lemma III.2 in \cite{bre}, and thus there is
$\tilde{C}$ such that for any $\varphi \in C(\R^d)$, we have
\begin{align*}
\intv{\varphi (v) F} & = \tilde{C} \into{\varphi (\omega) M(\omega) } = \rho
\frac{ \into{\varphi (\omega) \exp \left (\frac{\omega \cdot u[F]}{\sigma}   \right )}}{\into{\exp \left (\frac{\omega \cdot u[F]}{\sigma}   \right )}}
\end{align*}
with $\rho = \tilde{C}\into{M(\omega)}$.
Therefore the measure $F$ has a positive density with respect to $\dom$ on $\rsphere$
\[
F = \rho \frac{\exp \left (\frac{\omega \cdot u[F]}{\sigma}   \right )\;\dom
}{\intopr{\exp \left (\frac{\omega ^\prime \cdot u[F]}{\sigma}   \right )}}.
\]
If $\rho = 0$, we obtain $F = 0$, and we can take $l = 0$ and any $\Omega \in \sphere$. Assume now that $\rho >0$. If $u[F] = 0$, we obtain $F = \rho \frac{\dom}{\omega _d r ^{d-1}}$ which corresponds to $l = 0$ and any $\Omega \in \sphere$. If $u[F] \neq 0$, we introduce $\Omega [F] = \frac{u[F]}{|u[F]|}$. By the definition of $u[F]$, we have
\begin{equation}
\label{Equ20}
u[F] = \displaystyle\frac{
\into{\exp\left ( \frac{\omega \cdot u[F]}{\sigma}  \right )\omega}
}{\into{\exp\left ( \frac{\omega \cdot u[F]}{\sigma}  \right )}} =
\frac{
\intth{r \cos \theta \exp\left ( \frac{r |u[F]|}{\sigma} \cos \theta \right )\sin ^{d-2} \theta}}{\intth{\exp\left ( \frac{r |u[F]|}{\sigma} \cos \theta \right )\sin ^{d-2} \theta}}\Omega[F].
\end{equation}
For the last equality use the fact that
\[
\into{\exp\left ( \frac{\omega \cdot u[F]}{\sigma} \right)\omega } = \into{\exp\left ( \frac{\omega \cdot u[F]}{\sigma} \right)(\omega \cdot \Omega)}\;\Omega
\]
and formula \eqref{ForMag}. The equality \eqref{Equ20} reduces to the condition
\[
\frac{|u[F]|}{r} = \frac{
\intth{\cos \theta \exp\left ( \frac{r |u[F]|}{\sigma} \cos \theta \right )\sin ^{d-2} \theta}}{\intth{\exp\left ( \frac{r |u[F]|}{\sigma} \cos \theta \right )\sin ^{d-2} \theta}}.
\]
We introduce the function $\lambda : \R_+ \to \R$
\[
\lambda (l) = \frac{
\intth{ \cos \theta e^{l \cos \theta }\sin ^{d-2} \theta}}{\intth{e^{  l \cos \theta }\sin ^{d-2} \theta}},\;\;l \in \R_+.
\]
Therefore the non negative number $l=\frac{r |u[F]|}{\sigma}$
satisfies $\lambda (l) = \frac{\sigma}{r ^2} l$, and thus the
measure $F$ is given by
\[
F = \rho \frac{
\exp\left ( \frac{r |u[F]|}{\sigma} \frac{\omega}{r} \cdot \Omega \right )\dom}{\intopr{\exp\left ( \frac{r |u[F]|}{\sigma} \frac{\omega ^{\;\prime}}{r} \cdot \Omega\right )}} = \rho \mlo \;\dom
\]
with $\rho \in \R_+$, $\Omega = \frac{u[F]}{|u[F]|} \in \sphere$, $l \in \R_+$ satisfying $\lambda (l) = \frac{\sigma }{r^2} l$.\\
$2. \implies 1.$ Conversely, let $F$ be a measure given by $F = \rho \mlo \dom$ for some $\rho \in \R_+, \Omega \in \sphere, l \in \R_+$ such that $\lambda (l) = \frac{\sigma}{r^2}l$. If $\rho = 0$, $F$ is the trivial equilibrium (with $u[F] = 0$). If $\rho >0$, the mean velocity writes
\begin{align*}
u[F] & = \frac{\intv{v F }}{\intv{F}} = \frac{\into{(\omega \cdot \Omega)\exp \left (l \frac{\omega}{r} \cdot \Omega \right)}}{\into{\exp \left (l \frac{\omega}{r} \cdot \Omega   \right)}}\Omega \\
& = \frac{r\intth{\cos \theta e ^{l \cos \theta}\sin ^{d-2}\theta}}{\intth{e ^{l\cos \theta} \sin ^{d-2}\theta}}\Omega = r \lambda (l) \Omega = \frac{\sigma}{r} l \Omega
\end{align*}
saying that $\frac{u[F]}{|u[F]|} = \Omega$ and $|u[F]| = \frac{\sigma l}{r}$.
For any test function $\tpsi \in C^2  (\rsphere)$ we have
\begin{align*}
M(v) \left [(v-u[F]) \cdot \nabla _v \left [\tpsi\left ( r \vsv\right)   \right]  - \sigma \Delta _v \left [\tpsi\left ( r \vsv\right)   \right]\; \right] = - \sigma \Divo ( M\nabla _\omega \tpsi),\;\;v \in \rsphere
\end{align*}
where $M(v) = \exp \left ( - \frac{|v - u[F] |^2}{2\sigma}\right), v \in \R^d$. Notice that for any $v \in \rsphere$ we have
\begin{align*}
M(v) & = \exp \left ( - \frac{r^2 + \frac{\sigma ^2 l^2}{r^2}}{2\sigma}\right) \into{\exp \left ( l\Omega \cdot \frac{\omega}{r} \right ) } \; \mlo (\omega)
\end{align*}
and thus, the above equality becomes
\[
\mlo(v) \left \{ (v-u[F]) \cdot \nabla _v \left [\tpsi\left ( r \vsv\right)   \right]  - \sigma \Delta _v \left [\tpsi\left ( r \vsv\right)   \right] \right\} = - \sigma \Divo ( \mlo \nabla _\omega \tpsi).
\]
Therefore we obtain
\begin{align*}
\int_{v\neq 0} & \left \{ (v-u[F]) \cdot \nabla _v \left [\tpsi\left ( r \vsv\right)   \right]  - \sigma \Delta _v \left [\tpsi\left ( r \vsv\right)   \right] \right\}F\;\dv  \\
& = \int_{|v| = r} \left \{ (v-u[F]) \cdot \nabla _v \left [\tpsi\left ( r \vsv\right)   \right]  - \sigma \Delta _v \left [\tpsi\left ( r \vsv\right)   \right] \right\} \rho \mlo (v) \;\dv \\
& = - \rho \sigma \into{\Divo (\mlo(\omega) \nabla _\omega \tpsi )} = 0.
\end{align*}
\end{proof}
The properties of the function $\lambda$ are summarized in Proposition \ref{IsotropicEquilibrium}, whose proof is detalied below.
\begin{proof} (of Proposition \ref{IsotropicEquilibrium})\\
We introduce the function
\[
\beta _0 (l) = \frac{1}{\pi} \intth{e^{l \cos \theta} \sin ^{d-2} \theta},\;\;l \in \R.
\]
It is a Bessel like function \cite{AbraSteg}. Indeed, it verifies the linear second order differential equation
\begin{equation}
\label{EquBessel} l^2 \beta _0 ^{\prime \prime} (l) + (d-1) l \beta _0 ^\prime (l) = l^2 \beta _0 (l),\;\;l \in \R.
\end{equation}
We recall that the standard modified Bessel function $I_n (l) = \frac{1}{\pi} \intth{e ^{l\cos \theta} \cos (n \theta) }, n \in \N$, satisfy
\[
l^2 I_n ^{\prime \prime} (l) + l I_n ^\prime (l) = (l^2 + n^2 ) I_n (l),\;\;l \in \R.
\]
Clearly $\beta _0 ^\prime (l) = \frac{1}{\pi}\intth{\cos \theta e ^{l \cos \theta} \sin ^{d-2} \theta }$ and thus the function $\lambda $ writes
\[
\lambda (l) = \frac{\beta _0 ^\prime (l)}{\beta _0 (l)}.
\]
It is easily seen that $\beta _0 ^\prime (0) = 0$, implying that $\lambda (0) = 0$. Indeed, we have
\[
\pi \beta _0 ^\prime (0) = \intth{\cos \theta \sin ^{d-2} \theta } = \intth{\frac{\mathrm{d}}{\mathrm{d}\theta} \frac{\sin ^{d-1}\theta }{d-1} } = 0,\;\;d\geq 2.
\]
Moreover, $\lambda$ is strictly increasing. This comes by the formula
\begin{equation}
\label{EquDerLam} \lambda ^\prime (l) = \frac{\beta _0 ^{\prime \prime}(l) \beta _0 (l) - (\beta _0 ^\prime (l))^2}{\beta _0 ^2 (l)}
\end{equation}
and by observing that the Cauchy inequality implies
\begin{align*}
(\beta _0 ^\prime (l))^2& = \left (\frac{1}{\pi}\intth{\cos \theta e ^{l\cos \theta}\sin ^{d-2} \theta }    \right ) ^2 \\
& < \frac{1}{\pi} \intth{ e ^{l \cos \theta} \sin ^{d-2} \theta } \;\frac{1}{\pi} \intth{\cos ^2 \theta e ^{l \cos \theta}\sin ^{d-2} \theta } = \beta _0 (l) \beta _0 ^{\prime \prime} (l).
\end{align*}
The derivative of $\lambda $ at $l =0$ is
\begin{align*}
\lambda ^\prime (0) & =
\frac{
\beta _0 ^{\prime \prime} (0)}{\beta _0 (0)}=
\frac{\intth{\cos ^2 \theta \sin ^{d-2} \theta}}{\intth{\sin ^{d-2} \theta}}
= \frac{\intth{\cos \theta \frac{\mathrm{d}}{\mathrm{d}\theta} \frac{\sin ^{d-1} \theta }{d-1}}}{\intth{\sin ^{d-2} \theta}} \\
& = \frac{\intth{\sin ^d \theta}}{(d-1) \intth{\sin ^{d-2} \theta}}.
\end{align*}
Using $\sin ^2 \theta+\cos ^2 \theta=1$ in the first equality
above, we also have
\[
\lambda ^\prime (0) = 1 - \frac{\intth{\sin ^d
\theta}}{\intth{\sin ^{d-2} \theta }}\,.
\]
We deduce that
\[
\frac{\intth{\sin ^d \theta}}{\intth{\sin ^{d-2} \theta }} = 1 - \lambda ^\prime (0) = (d-1) \lambda ^\prime (0)
\]
which yields $\lambda ^\prime (0) = 1/d$. We claim that $\lambda $
is strictly  concave. Combining \eqref{EquDerLam} and
\eqref{EquBessel}, we obtain for any $l>0$
\begin{equation}
\label{Equ22}
\lambda ^\prime (l) = \frac{\left ( \beta _0 (l) - \frac{d-1}{l} \beta _0 ^\prime (l)\right ) \beta _0 (l)}{\beta _0 ^2 (l)} - \left ( \frac{\beta _0 ^\prime (l)}{\beta _0 (l)}\right ) ^2 = 1 - \frac{d-1}{l} \lambda (l) - \lambda ^2 (l).
\end{equation}
As $\lambda $ is positive and strictly increasing, we deduce that $\lambda $ is strictly concave on $\R_+$. Clearly the function $\lambda$ is bounded on $\R_+$
\[
0 = \lambda (0) < \lambda (l) = \frac{\intth{\cos \theta e ^{l \cos \theta} \sin ^{d-2} \theta}}{\intth{e ^{l \cos \theta} \sin ^{d-2} \theta}}<1
\]
and $\frac{1}{d} = \lambda ^\prime (0) > \lambda ^\prime (l) > 0, l>0$. Let us denote by $\Lambda _0, \Lambda _1$ the limits
\[
\Lambda _0 = \lim _{l \to +\infty} \lambda (l) \in ]0,1],\;\;\Lambda _1 = \lim _{l \to +\infty} \lambda ^\prime (l) \in [0, \frac{1}{d}[.
\]
If $\Lambda _1 >0$ then the inequality $\lambda ^\prime (l) >\Lambda _1, l>0$, implies
\[
\lim _{l \to +\infty} \lambda (l) = \lim _{l \to +\infty} \{ \lambda (l) - \lambda (0)\} \geq \lim _{l \to +\infty} l \Lambda _1  = +\infty
\]
which contradicts the boundedness of $\lambda$. Therefore $\Lambda _1 = 0$ and thus $\lambda ^\prime ([0,+\infty[) = ]0,\lambda ^\prime (0)] = ]0,1/d]$. Passing to the limit, when $l \to +\infty$, in \eqref{Equ22}, yields $\Lambda _0 = \lim _{l \to +\infty} \lambda (l) = 1$.\\
If $\frac{\sigma}{r^2} \geq \frac{1}{d}$, the function $l \to \lambda (l) - \frac{\sigma }{r^2} l$ is strictly decreasing on $\R_+$, and vanishes at $l = 0$
\[
\lambda ^\prime (l) - \frac{\sigma }{r^2} < \lambda ^\prime (0)  - \frac{\sigma }{r^2} = \frac{1}{d} - \frac{\sigma}{r^2} \leq 0,\;\;l>0
\]
implying that the only solution of $\lambda (l) = \frac{\sigma
}{r^2}l $ on $\R_+$ is $l = 0$. If $\frac{\sigma}{r^2}\in ]0,
\frac{1}{d}[$, there is a unique $\tilde{l}>0$ such that $\lambda
^\prime (\tilde{l}) = \frac{\sigma}{r^2}$ and the function $l \to
\lambda ^\prime (l) - \frac{\sigma}{r^2}$ is positive on
$]0,\tilde{l}[$ and negative on $]\tilde{l}, +\infty[$. Therefore
the function $l \to \lambda (l) - \frac{\sigma}{r^2} l$ is
strictly increasing on $[0,\tilde{l}]$, strictly decreasing on
$[\tilde{l}, +\infty[$
\[
\left \{\lambda (l) - \frac{\sigma}{r^2} l\right \} |_{l = 0} = 0,\;\;\lim _{l \to +\infty} \left \{ \lambda (l) - \frac{\sigma}{r^2}l \right \} = - \infty.
\]
We deduce that there is a unique solution $l>0$ such that
$\lambda(l) = \frac{\sigma}{r^2}l$.
\end{proof}
\begin{remark}
The value $l = 0$ corresponds to the isotropic equilibrium
$M_{0\Omega} \;\dom = \frac{\dom}{\bar{\omega}_d \,r^{d-1}}$. The
limit when $l \to +\infty$ leads to the Dirac measure on
$\rsphere$, concentrated at $r \Omega$, that is, for any function
$\tpsi \in C(\rsphere)$ we have
\[
\lim _{l \to +\infty} \into{\tpsi (\omega) \mlo (\omega) } = \tpsi (r \Omega).
\]
\end{remark}
The function $\lambda$ can be computed explicitly, at least for $d
= 3$. Nevertheless, very good explicit approximations are
available in any dimension $d$.
\begin{lemma}
\label{Approximation}$\;$
\begin{enumerate}
\item
Consider the function
\[
\mu : \R_+ \to \R_+, \;\;\mu (l)= \frac{\sqrt{d^2 + 4l^2}-d}{2l} = \frac{2l}{ \sqrt{d^2 + 4l^2} + d},\;\;l \in \R_+.
\]
The function $\mu$ is strictly increasing, strictly concave and we have
\[
\mu (0) = \lambda (0) = 0,\;\;\mu ^\prime (0) = \lambda ^\prime (0) = \frac{1}{d},\;\;\lim _{l \to +\infty} \mu (l) = 1
\]
\[
\mu ^\prime (l) < 1 - \frac{d-1}{l} \mu (l) - \mu ^2 (l),\;\;\mu (l) < \lambda (l), \;\;l>0.
\]
\item
If $d=3$, the function $\lambda$ is given by $\lambda (l) = \frac{\cosh(l)}{\sinh(l)} - \frac{1}{l}, l >0$.
\end{enumerate}
\end{lemma}
\begin{proof}$\;$\\
1. By direct computations we obtain
\[
\mu ^\prime (l) = \frac{2d}{\sqrt{d^2 + 4l^2} ( \sqrt{d^2 + 4l^2} + d)}>0,\;\;l \in \R_+
\]
and
\[
1 - \frac{d-1}{l} \mu (l) - \mu ^2 (l) = \frac{2}{\sqrt{d^2 + 4l^2} + d}.
\]
Therefore $\mu$ satisfies the first order differential inequation
\[
\mu ^\prime (l) = \frac{2d}{\sqrt{d^2 + 4l^2} ( \sqrt{d^2 + 4l^2} + d)}< \frac{2}{\sqrt{d^2 + 4l^2} + d} = 1 - \frac{d-1}{l} \mu (l) - \mu ^2 (l),\;\;l >0
\]
and the initial condition $\mu (0) = 0$. Recall that $\lambda $ satisfies the first order differential equation (cf. \eqref{Equ22})
\[
\lambda ^\prime (l) = 1 - \frac{d-1}{l} \lambda (l) - \lambda ^2 (l),\;\;l >0
\]
with the initial condition $\lambda (0) = 0$. By comparison principle, it follows that $\mu (l) < \lambda (l)$ for any $l >0$. Clearly $\mu ^\prime (0) = \frac{1}{d} = \lambda ^\prime (0), \lim _{l \to +\infty } \mu (l) = 1$, $\mu ^\prime (l) >0, l \in \R_+$, and $\mu ^\prime $ is strictly decreasing, saying that $\mu$ is strictly increasing and strictly concave on $\R_+$. \\
2. In the case $d = 3$ we obtain
\[
\pi \beta _0 (l) = \intth{e^{l \cos \theta } \sin \theta } = \frac{e^l - e^{-l}}{l},\;\;l>0
\]
\[
\pi \beta _0 ^\prime (l) = \intth{e^{l \cos \theta} \cos \theta \sin \theta} = \frac{e^l + e^{-l}}{l} - \frac{e^l - e ^{-l}}{l^2},\;\;l>0
\]
implying that
\[
\lambda (l) = \frac{\beta _0 ^\prime (l)}{\beta _0 (l)} = \frac{\cosh(l)}{\sinh(l)} - \frac{1}{l}, l >0.
\]
\end{proof}
In order to exploit the constraint \eqref{Equ9} we will need to compute $Q(F)$, where $F$ is a von Mises-Fisher equilibrium, let us say $F = \mlo (\omega)\dom$. This computation is detailed in the following lemma. The notation $(\cdot, \cdot)$ stands for the pairing between distributions and smooth functions.
\begin{lemma}
Let $F = \mlo (\omega)\dom$ be a von Mises-Fisher equilibrium. Then we have, for any function $\varphi \in C^2 _c (\R^d)$
\[
(Q(F), \varphi) = \sigma \frac{\mlo}{M}
\frac{\mathrm{d}}{\mathrm{d}t}_{|_{t=r}} \int _{|\omega _t| = t}
M(\omega _t) (\nabla _v \varphi )(\omega _t) \cdot \frac{\omega
_t}{t} \;\mathrm{d}\omega _t
\]
where $M(v) = \exp \left ( - \frac{|v-u[F]|^2}{2\sigma}\right ), v \in \R^d$.
\end{lemma}
\begin{proof}
Pick a test function $\varphi \in C^2 _c (\R^d)$ and notice that
\begin{align*}
(Q(F), \varphi) & = ( F, \sigma \Delta _v \varphi - (v- u[F]) \cdot \nabla _v \varphi)\\
& = \left ( F, \sigma \frac{\Divv ( M \nabla _v \varphi)}{M(v)}\right )\\
& = \sigma \into{\Divv(M \nabla _v \varphi) (\omega)
\frac{\mlo(\omega)}{M(\omega)}}.
\end{align*}
It is easily seen that the function $\frac{\mlo}{M}$ is constant on the sphere $\rsphere$
\[
\frac{\mlo(\omega)}{M(\omega)} = \frac{\exp\left (\frac{r^2 + |u[F]|^2 }{2\sigma}   \right)}{\intopr{\exp \left (l \Omega \cdot \frac{\omega ^\prime}{r}\right )}},\;\;\omega \in \rsphere
\]
and therefore we have
\begin{align*}
(Q(F), \varphi) & = \sigma \frac{\mlo}{M} \frac{\mathrm{d}}{\mathrm{d}t}_{|_{t = r}} \int _{|v|< t} \Divv ( M \nabla _v \varphi )\;\mathrm{d}v \\
& = \sigma \frac{\mlo}{M}\frac{\mathrm{d}}{\mathrm{d}t}_{|_{t =
r}}\int _{|\omega _t | = t} M (\omega _t) \nabla _v \varphi
(\omega _t) \cdot \frac{\omega _t}{t} \;\mathrm{d}\omega _t.
\end{align*}
\end{proof}
Thanks to the above result, we can determine $\Fo - \ave{\Fo}$ in terms of $F$. More exactly we prove
\begin{lemma}
\label{FOneF}
Let $F = \mlo(\omega)\dom$ be a von Mises-Fisher equilibrium and $\Fo$ a bounded measure such that
\[
\Divv\{\Fo (\alpha - \beta |v|^2 ) v\} = Q(F).
\]
Then for any function $\chi \in \coc$, such that $\chi |_{\rsphere} = 0$ we have
\begin{align*}
\intv{\chi (v) \left ( \Fo - \ave{\Fo}\right )} & = \int_{v \neq 0} \chi (v) \Fo \;\dv  \\
& = \sigma \frac{\mlo}{M} \frac{\mathrm{d}}{\mathrm{d}t}_{|_{t =
r}} \int _{|\omega _t| = t}  \frac{M(\omega_t)\chi (\omega _t)}{t
\beta (t^2 - r^2)}\;\mathrm{d}\omega _t.
\end{align*}
\end{lemma}
\begin{proof}
For any function $\varphi \in \coc$, we know that
\begin{align*}
- \intv{\abv \cdot \nabla _v \varphi \;\Fo} & = (Q(F), \varphi) \\
& = \sigma \frac{\mlo}{M} \frac{\mathrm{d}}{\mathrm{d}t}_{|_{t =
r}} \int _{|\omega _t| = t} M (\omega _t) \nabla _v \varphi
(\omega _t) \cdot \frac{\omega _t}{t} \;\mathrm{d}\omega _t.
\end{align*}
The idea is to solve the adjoint problem (cf. Lemma \ref{AdjointProblem})
\begin{equation*}
\label{EquAdjProb}
- (\alpha - \beta |v|^2) v \cdot \nabla _v \varphi = \chi (v)
\end{equation*}
and to express the normal derivative of $\varphi$ in terms of $\chi$. Indeed, for any $\omega _t \in t \sphere$, we have
\[
\nabla _v \varphi (\omega _t) \cdot \frac{\omega _t}{t} = \frac{\chi (\omega _t)}{t (\beta t^2 - \alpha)} = \frac{\chi (\omega _t)}{t \beta (t^2 - r^2)}.
\]
Finally we obtain the formula
\[
\int _{v \neq 0} \chi (v) \Fo\;\dv = (Q(F), \varphi) = \sigma
\frac{\mlo}{M} \frac{\mathrm{d}}{\mathrm{d}t}_{|_{t = r}} \int
_{|\omega _t| = t}  \frac{M(\omega_t)\chi (\omega _t)}{t \beta
(t^2 - r^2)}\;\mathrm{d}\omega _t.
\]
\end{proof}
Once we have determined the form of the dominant distribution
$f(t,x,v) = \rho (t,x) M_{l\Omega (t,x)} \dom$, we search for
macroscopic equations characterizing $\rho (t,x)$ and $\Omega
(t,x)$. For doing that, we use the moments of \eqref{Equ12} with
respect to the velocity. The key point is how to eliminate $\fo$
in the right hand side of \eqref{Equ12}. Notice that this right
hand side is the linearization around $f$, with $\intv{f} >0$,
computed in the direction $\fo$, of the average collision kernel
$Q$
\begin{align*}
\mathcal{L}_f(f^{(1)}):=\lime\frac{\ave{Q(f + \eps \fo)} - \ave{Q(f)}}{\eps} = &\, \ave{\Divv \left [\fo (v-u[f]) + \sigma \nabla _v \fo   \right ]}\\
& - \ave{\Divv \left [f \frac{\intvp{\fo (v ^\prime - u[f])}}{\intvp{f}}    \right] }\\
= &\, \ave{\Divv A_f (\fo)}
\end{align*}
where
\[
A_f (\fo) = \left [\fo (v-u[f]) + \sigma \nabla _v \fo   \right ]-
f \frac{\intvp{\fo (v ^\prime - u[f])}}{\intvp{f}}.
\]
We are looking for functions such that
\begin{equation}
\label{Equ23} \intv{\psi (v)\ave{\Divv A_f (\fo)}}
\end{equation}
can be expressed in terms of the velocity moments of $f$, in order
to get a closure for the macroscopic quantities $\rho (t,x),
\Omega (t,x)$. For example $\psi (v) = 1$ leads to the continuity
equation
\[
\partial _t \intv{f} + \Divx  \intv{v f}  = 0
\]
which also writes
\[
\partial _t \rho + \Divx \left (\rho \frac{\sigma}{r} l \Omega  \right )= 0.
\]
Naturally, we need to find other functions $\psi$, which will
allow us to characterize the time evolution of the orientation
$\Omega$. Recall that the constraint \eqref{Equ9} determines $\fo
- \ave{\fo}$ (in terms of $f$), but not $\ave{\fo}$, as Lemma
\ref{FOneF} implies. Motivated by this, we are looking for
functions $\psi$ such that
\begin{equation*}
\label{Equ24} \intv{\psi (v) \ave{\Divv A_f (\go)}} = 0
\end{equation*}
for any measures $f, \go$ supported in $\R^d \times \rsphere$. Indeed, in that case the expression in \eqref{Equ23} can be computed in terms of $f$, provided that we neglect the mass of $\fo$ at $\R^d \times \{0\}$
\begin{align*}
\intv{\psi (v) \ave{\Divv A_f (\fo)}} = &\,\intv{\psi  \ave{\Divv A_f \ave{\fo}}} \\
& +\intv{\psi  \ave{\Divv A_f \left [ \fo - \ave{\fo}\right ]}}  \\
= &\, \intv{\psi (v) \ave{\Divv A_f \left [ \fo - \ave{\fo}\right
]}}.
\end{align*}
Let us concentrate now on the collision invariants of the average
collision operator. Recall that the linearized of $\ave{Q}$,
around a measure $F$ such that $\intv{F}>0$, writes
\[
\lime \frac{\ave{Q(F + \eps \Fo)} - \ave{Q(F)}}{\eps} = \ave{\Divv A_F (\Fo)}
\]
where
\[
A_F (\Fo) = \left [\Fo (v-u[F]) + \sigma \nabla _v \Fo \right] -
F \frac{\intvp{\Fo (v^\prime - u[F])}}{\intvp{F}}.
\]
We search for functions $\psi = \psi (v)$ such that
\begin{equation}
\label{Equ41}
\intv{\psi (v) \ave{\Divv A_F (\Go)}} = 0
\end{equation}
for any bounded measures $F, \Go$ supported in $\rsphere$.
Actually, since we already know that the dominant term is a von
Mises-Fisher distribution, it is enough to impose \eqref{Equ41}
only for $F = \mlo\dom$, with $\lambda (l) = \frac{\sigma}{r^2}
l$, for some given $\Omega \in \sphere$. Doing that, to any
orientation $\Omega$, we associate a family of suitable
pseudo-collision invariants, allowing us to determine the
macroscopic equations satisfied by the moments $\rho, \Omega$. A
similar construction was done in \cite{DM08}, baptized as
generalized collision invariants. Even if our approach is not
exactly the same as in \cite{DM08}, we will continue referring to
them as generalized collision invariants. Notice that once we have
determined $\psi$ such that \eqref{Equ41} is verified for any
bounded measure $\Go$ supported in $\rsphere$, we need to check
that \eqref{Equ41} still holds true for any bounded measure, not
necessarily supported in $\rsphere$, satisfying the constraint
\eqref{Equ9} (see Proposition \ref{ZeroAvePart} and \ref{C}). The
condition \eqref{Equ41} should be understood in the following
sense
\[
\int_{v \neq 0} \tpsi \left ( r \vsv\right) \Divv \{A_F (\Go)\}\;\dv = 0,\;\;F = \mlo \;\dom
\]
for any $\Go \in {\mathcal M}_b (\R^d)$, $\supp \Go \subset \rsphere$, that is
\begin{align}
\label{Equ42}
& \int_{v\neq 0} \left \{- (v - u[F]) \cdot \nabla _v \left[\tpsi \left ( r \vsv \right) \right]  + \sigma \Delta _v \left[\tpsi \left ( r \vsv \right) \right]\right \}\Go\;\dv \nonumber \\
& + \int_{v \neq 0} \frac{\int_{v^\prime  \neq 0} (v ^\prime - u[F])\Go\;\mathrm{d}v^\prime}{\intvp{F}} \cdot \nabla _v \left[\tpsi \left ( r \vsv \right) \right] F\;\dv = 0
\end{align}
for $F = \mlo\dom$ and any $\Go \in {\mathcal M}_b (\R^d)$, $\supp \Go \subset \rsphere$. Taking into account the equalities
\[
\nabla _v \left [\tpsi \left ( r \vsv \right )  \right ] = \nabla _\omega \tpsi,\;\;\Delta _v \left [\tpsi \left ( r \vsv \right )  \right ] = \Delta _\omega \tpsi,\;\;|v| = r
\]
the condition \eqref{Equ42} becomes
\begin{align}
\label{Equ43} (\omega - u[\mlo] ) \cdot \nabla _\omega \tpsi
-\sigma \Delta_\omega \tpsi = ( \omega - u[\mlo]) \cdot
\frac{\intopr{\nabla _{\omega ^\prime} \tpsi \mlo
}}{\intopr{\mlo}}= 0.
\end{align}


\section{The generalized collision invariants}
\label{ColInv}
In this section, we concentrate on the resolution
of the linear equation \eqref{Equ43}. If we introduce the vector
\begin{equation*}
\label{Equ61}
W[\tpsi] = \frac{\into{\nabla _\omega \tpsi \mlo(\omega)}}{\into{\mlo (\omega)}} = \into{\nabla _\omega \tpsi \mlo(\omega)}
\end{equation*}
the equation \eqref{Equ43} becomes elliptic on $\rsphere$ and
reads
\begin{equation}
\label{Equ62}
- \sigma \Divo ( \mlo \nabla _\omega \tpsi ) = \mlo (\omega) ( \omega - u[\mlo]) \cdot W[\tpsi].
\end{equation}
Any solution of equation \eqref{Equ62} will be called a
generalized collision invariant of the average collision operator
$\ave{Q}$.

The solvability of \eqref{Equ62} requires that the integral of the
right hand side over $\rsphere$ vanishes, i.e.,
$$
\into{\mlo (\omega) ( \omega - u[\mlo]) \cdot W[\tpsi]}=0
$$
which is true, by the definition of the mean velocity $u[\mlo]$.
But there is another compatibility condition to be fullfiled. Take
any vector $\Wp \in \R^d$ and multiply the equation \eqref{Equ62}
by the scalar function $\omega \to \Wp \cdot \omega$, whose
gradient along $\rsphere$ is $\imoo \Wp$. Integrating by parts
yields
\[
\sigma \!\into{\!\!\!\!\mlo (\omega) \nabla _\omega \tpsi } \cdot
\Wp = \into{\! \!\!\!\mlo (\omega) (\omega - u[\mlo])\otimes (
\omega - u[\mlo]) } : W[\tpsi]\otimes \Wp
\]
saying that $W[\tpsi]$ is an eigenvector of the matrix
\[
\calmlo := \into{\mlo (\omega) (\omega - u[\mlo])\otimes ( \omega - u[\mlo])}
\]
corresponding to the eigenvalue $\sigma$. The following lemma details the spectral properties of the matrix $\calmlo$.

\begin{lemma}
\label{SpecProp}
For any $l \in \R_+$ such that $\lambda (l) = \frac{\sigma}{r^2}l$, and $\Omega \in \rsphere$, the matrix $\calmlo$ is symmetric, definite positive and
\[
{\mathcal M}_{0\Omega} = \frac{r^2}{d}I_d,\;\;\calmlo = (r^2 - (d-1) \sigma - |u|^2) \Omega \otimes \Omega + \sigma (I_d - \Omega \otimes \Omega),\;\;l>0,\;\;0 < \frac{\sigma}{r^2} < \frac{1}{d}.
\]
If $0 < \frac{\sigma}{r^2} < \frac{1}{d}$, we have $r^2 - (d-1) \sigma - |u|^2 < \sigma$ and, in particular $\ker ( \calmlo - \sigma I_d) = (\R \Omega) ^\bot$.
\end{lemma}
\begin{proof}
Clearly $\calmlo$ is symmetric and definite positive. The case $l
= 0$ is trivial, and we have ${\cal M}_{0\Omega} =
\frac{r^2}{d}I_d$. Assume now that $l>0$ and thus necessarily
$\frac{\sigma}{r^2} \in ]0,\frac{1}{d}[$ cf. Proposition
\ref{IsotropicEquilibrium}. We consider a orthonormal basis
$\{E_1, ..., E_{d-1}, \Omega\}$. It is easily seen that
\begin{align*}
\calmlo & = \into{ (\omega - u) \otimes \omega \mlo} \\
& = \into{ \! \left[((\omega \cdot \Omega) - |u|)\Omega + \sum _{i = 1}^{d-1} (\omega \cdot E_i)E_i\right] \!\!\otimes\!\!
            \left[(\omega \cdot \Omega)\Omega + \sum _{i = 1}^{d-1} (\omega \cdot E_i)E_i\right]\!\mlo\!}\\
& = \into{\!\!\!\!\!\!\! ((\omega \cdot \Omega) - |u|)(\omega \cdot \Omega)\mlo }\;\Omega \otimes \Omega + \sum _{i = 1}^{d-1} \into{ (\omega \cdot E_i)^2 \mlo} \;E_i \otimes E_i \\
& = \into{ \!\!\!\!\!\!\!\!\!((\omega \cdot \Omega)^2 - |u|^2)\mlo }\;\Omega \otimes \Omega + \into{\!\!\!\!\!\!\!\!\frac{(r^2 - (\omega \cdot \Omega )^2 )}{d-1}\mlo} (I_d - \Omega \otimes \Omega).
\end{align*}
We show that
\[
\into{(\omega \cdot \Omega)^2 \mlo } = r^2 - (d-1) \sigma.
\]
This comes by the condition $\lambda (l) = \frac{\sigma}{r^2}l$ and integrations by parts
\begin{align*}
r^2 - \into{(\omega \cdot \Omega)^2 \mlo } & = \frac{r^2\intth{e^{l \cos \theta} \sin^d \theta}}{\intth{ e ^{l \cos \theta}\sin ^{d-2} \theta}}\\
& = -\frac{r^2}{l} \frac{\intth{\frac{\mathrm{d}}{\mathrm{d}\theta}e^{l\cos \theta}\sin ^{d-1} \theta}}{\intth{ e ^{l \cos \theta}\sin ^{d-2} \theta}}\\
& = (d-1) \frac{r^2}{l}\frac{\intth{ \cos \theta e ^{l \cos \theta}\sin ^{d-2} \theta}}{\intth{ e ^{l \cos \theta}\sin ^{d-2} \theta}}\\
& = (d-1) \frac{r^2}{l}\lambda (l)\\
& = (d-1) \frac{r^2}{l}\frac{\sigma}{r^2}l = (d-1) \sigma.
\end{align*}
We deduce also that
\begin{align*}
\into{ ((\omega \cdot \Omega)^2 - |u|^2)\mlo } & = r^2 - (d-1) \sigma - |u|^2
\end{align*}
and therefore
\[
\calmlo = (r^2 - (d-1) \sigma - |u|^2) \Omega \otimes \Omega + \sigma (I_d - \Omega \otimes \Omega).
\]
We claim that the biggest eigenvalue is $\sigma$, that is $r^2 - (d-1) \sigma - |u|^2 < \sigma$, or equivalently $r^2 < d\sigma + |u|^2$. This is a consequence of Lemma  \ref{Approximation}. Indeed, since $l>0$, we know that
\[
\mu (l) = \frac{2l}{\sqrt{d^2 + 4l^2} + d} < \lambda (l) = \frac{\sigma}{r^2} l
\]
implying that
\[
\sqrt{d^2 + 4l^2} > \frac{2r^2}{\sigma} - d >0,\;\;\mbox{since } r^2 > d \sigma
\]
or equivalently
\[
4l^2 > 4 \frac{r^4}{\sigma ^2} - 4 d \frac{r^2}{\sigma}.
\]
Replacing $l = \frac{|u|r}{\sigma}$ in the above inequality, yields $r^2 < d \sigma + |u|^2$.
\end{proof}
The resolution of \eqref{Equ43} follows immediately, thanks to Lemma \ref{SpecProp}. As \eqref{Equ43} is linear and admits any constant function on $\rsphere$ as solution, we will work with zero mean solutions on $\rsphere$, that is $\into{\tpsi (\omega)} = 0$.
\begin{pro}
\label{CollInvBis}
Let $\mlo$ be a von Mises-Fisher distribution {\it i.e.,} $\Omega \in \sphere, l \in \R_+, \lambda (l) = \frac{\sigma}{r^2}l$, and $E_1, ..., E_{d-1}$ be a orthonormal basis of $(\R \Omega)^\bot$.
\begin{enumerate}
\item
If $l = 0$ and $\frac{\sigma}{r^2} \neq \frac{1}{d}$, then the only (zero mean) solution of \eqref{Equ43} is the trivial one.

\item If $l = 0$ and $\frac{\sigma}{r^2} = \frac{1}{d}$, then the
family of zero mean solutions for \eqref{Equ43} is a linear space
of dimension $d$. A basis is given by the functions $\tpsi _1,
...,\tpsi _d$ satisfying
\begin{equation}
\label{Equ63} - \sigma \Divo (M_{0\Omega} \nabla _\omega \tpsi _i)
= M_{0\Omega} (\omega) (\omega \cdot E_i),\;\into{\tpsi _i
(\omega)} = 0,
\end{equation}
for $i \in \{1, ..., d\}$ and $E_d = \Omega$.

\item
If $0 < \frac{\sigma}{r^2}<\frac{1}{d}, l>0, \lambda (l) = \frac{\sigma}{r^2}l$, then the family of zero mean solutions for \eqref{Equ43} is a linear space of dimension $d-1$. A basis is given by the functions $\tpsi _1, ..., \tpsi _{d-1}$ satisfying
\begin{equation}
\label{Equ65} - \sigma \Divo ( \mlo \nabla _\omega \tpsi _i) =
\mlo (\omega) ( \omega \cdot E_i),\;\;\into{\tpsi _i (\omega)} =
0,
\end{equation}
for $i \in \{1, ..., d-1\}$.
\end{enumerate}
\end{pro}
\begin{proof}$\;$\\
1. Let $\tpsi$ be a zero mean solution of \eqref{Equ43}. Multiplying by $(\omega \cdot \Wp)$, with $\Wp \in \R^d$, and integrating by parts over $\rsphere$ yield
\begin{align*}
\sigma W[\tpsi] \cdot \Wp & = \sigma \into{M_{0\Omega} \nabla _\omega \tpsi \cdot \Wp } = \into{M_{0\Omega}(\omega -0)\cdot W[\tpsi](\omega \cdot \Wp) } \\
& = {\mathcal M}_{0\Omega} W[\tpsi]\cdot \Wp = \frac{r^2}{d} W[\tpsi]\cdot \Wp.
\end{align*}
Therefore $\left ( \sigma - \frac{r^2}{d}\right ) W[\tpsi] = 0$, implying that $W[\tpsi] = 0$ and
\[
-\Divo (M_{0\Omega} (\omega) \nabla _\omega \tpsi ) = 0.
\]
We deduce that $\tpsi$ is a constant, zero mean function on $\rsphere$, and thus $\tpsi = 0$. \\
2. As $l=0$, then $\into{\omega M_{0\Omega}(\omega)} = u = 0$. Therefore the right hand sides in \eqref{Equ63} are zero mean functions on $\rsphere$, and by Lax-Milgram lemma, the zero mean functions $(\tpsi _i)_{1\leq i\leq d}$ are well defined. Notice that these functions also solve \eqref{Equ43}. Indeed, after multiplication by $(\omega \cdot \Wp)$, with $\Wp \in \R^d$, and integration by parts we obtain, for any $i \in \{1,...,d\}$
\[
\sigma \into{\nabla _\omega \tpsi _i \cdot \Wp M_{0\Omega}} = \into{(\omega \cdot E_i) (\omega \cdot \Wp) M_{0\Omega}}=  {\mathcal M}_{0\Omega} E_i \cdot \Wp.
\]
We deduce that
\begin{equation}
\label{Equ66}
\sigma \into{M_{0\Omega} (\omega) \nabla _\omega \tpsi _i } = {\mathcal M}_{0\Omega} E_i = \frac{r^2}{d} E_i = \sigma E_i,\;\;i \in \{1, .., d\}
\end{equation}
which eactly says that $(\tpsi_i)_{1\leq i \leq d}$ solve \eqref{Equ43}. It is easily seen that the family $(\tpsi _i)_{1\leq i \leq d}$ is linearly independent~: if $\sum_{i = 1}^d c_i \tpsi _i = 0$, then by \eqref{Equ66} one gets
\[
\sum_{i=1} ^d c_i E_i = \sum_{i=1}^d c_i \into{M_{0\Omega} (\omega) \nabla _\omega \tpsi _i} = 0
\]
implying that $c_i= 0, i \in \{1,...,d\}$. We show now that any zero mean solution $\tpsi $ for \eqref{Equ43} is a linear combination of $(\tpsi _i)_{1\leq i \leq d}$. Let $(c_i)_{1\leq i \leq d}$ be the coordinates of the vector $W[\tpsi]$ with respect to the basis $(E_i)_{1\leq i \leq d}$
\[
W[\tpsi] = \into{M_{0\Omega} (\omega) \nabla _\omega \tpsi } = \sum _{i = 1} ^d c_i E_i.
\]
We claim that $\tpsi = \sum _{i = 1} ^d c_i \tpsi _i$. Indeed, since $\tpsi$ and $\sum _{i = 1} ^d c_i \tpsi _i$ have zero mean, thanks to the uniqueness of zero mean solution, it is enough to check that $\sum _{i = 1} ^d c_i \tpsi _i$ solves \eqref{Equ62}, with the right hand side $M_{0\Omega} \omega \cdot W[\tpsi]$. Indeed, we have
\[
- \sigma \Divo \left( M_{0\Omega} \nabla _\omega \sum _{i = 1} ^d
c_i \tpsi _i\right) = \sum _{i = 1}^d c_i M_{0\Omega} (\omega
\cdot E_i) = M_{0\Omega} (\omega -0)\cdot W[\tpsi]
\]
implying that $\tpsi = \sum _{i = 1} ^d c_i \tpsi _i$.\\
3. The arguments are similar. The solutions $(\tpsi _i)_{1\leq i \leq d-1}$ in \eqref{Equ65} also solve \eqref{Equ43}, and are linearly independent. But for any solution $\tpsi$ of \eqref{Equ43}, we have for any $\Wp \in \R^d$
\begin{align*}
\sigma W[\tpsi] \cdot \Wp & = \sigma \into{\mlo{} \nabla _\omega \tpsi \cdot \Wp }= \into{\!\!\!\!\mlo (\omega - u[\mlo]) \cdot W[\tpsi] (\omega \cdot \Wp)} \\
& = \calmlo W[\tpsi] \cdot \Wp.
\end{align*}
Therefore $W[\tpsi] \in \ker ( \calmlo - \sigma I_d) = (\R \Omega)^\bot = \mathrm{span}\{E_1, ..., E_{d-1}\}$ and we deduce that $\tpsi = \sum _{i = 1} ^{d-1} c_i \tpsi _i$, with $W[\tpsi] = \sum_{i = 1} ^{d-1}c_i E _i$.
\end{proof}
We focus now on the structure of the solutions of \eqref{Equ43}. This is a consequence of the symmetry of $\mlo{}$, by rotations leaving invariant the orientation $\Omega$. We concentrate on the case $0 < \frac{\sigma}{r^2} < \frac{1}{d}, \lambda (l) = \frac{\sigma}{r^2}l, l>0$.
\begin{pro}
\label{Struct1}
For any $W \in \R^d, W \cdot \Omega = 0$, let us denote by $\tpsi _W$ the unique solution of the problem
\[
- \sigma \Divo (\mlo \nabla _\omega \tpsi _W) = \mlo \;(\omega - u)\cdot W = \mlo \; (\omega \cdot W),\;\;\into{\tpsi _W } = 0.
\]
For any orthogonal transformation $\calo$ of $\R^d$, leaving invariant the orientation $\Omega$, that is $\calo \Omega = \Omega$, we have
\[
\tpsi _W (\calo \omega) = \tpsi _{^t \calo W } (\omega),\;\;\omega \in \rsphere.
\]
\end{pro}
\begin{proof}
We know that $\tpsi _W$ is the minimum point of the functional
\[
J_W (z) = \frac{\sigma}{2}\into{\mlo |\nabla _\omega z|^2 } - \into{\mlo (\omega \cdot W) z(\omega) }
\]
on $z\in H^1(\rsphere),\into{z(\omega)} = 0$. It is easily seen that, for any orthogonal transformation ${\cal O}$ of $\R^d$, and any function $z \in H^1 (\rsphere)$, $\into{z(\omega)} = 0$, we have
\[
z_{\calo}: = z\circ \calo \in H^1(\rsphere),\;\;\into{z_{\calo}(\omega)} = 0
\]
and
\[
(\nabla _\omega z_{\calo})(\omega) = \;^t \calo(\nabla _\omega z)(\calo \omega),\;\;\omega \in \rsphere.
\]
Moreover, for any $z \in H^1(\rsphere), \into{z(\omega)} = 0$, and any orthogonal transformation leaving invariant the orientation $\Omega$ we obtain
\begin{align*}
J_{^t\calo W} (z_{\calo}) & = \frac{\sigma}{2}\into{\!\!\!\!\!\mlo(\omega) |^t \calo (\nabla _\omega z) (\calo \omega)|^2 } - \into{\!\!\!\!\!\!\!\!\mlo (\omega) ( \omega \cdot \;^t \calo W) z (\calo \omega)} \\
& = \frac{\sigma}{2}\into{\!\!\!\!\!\mlo(\calo\omega) | (\nabla _\omega z) (\calo \omega)|^2 } - \into{\!\!\!\!\!\!\mlo (\calo\omega) ( \calo \omega \cdot  W) z (\calo \omega)} \\
& = \frac{\sigma}{2}\into{\mlo(\omega) |\nabla _\omega z (\omega)|^2 } - \into{\mlo (\omega) ( \omega \cdot  W) z (\omega)} \\
& = J_W(z).
\end{align*}
Finally, one gets for any $z \in H^1(\rsphere), \into{z(\omega)} = 0$
\[
J_{^t \calo W} (\tpsi _W \circ \calo) = J_W (\tpsi _W ) \leq J_W(z \circ \;^t\calo) = J_{^t \calo W} (z)
\]
saying that $\tpsi _W \circ \calo = \tpsi _{^t \calo W}$.
\end{proof}
We claim that there is a function $\chi$ such that, for any $i \in \{1, ..., d-1\}$, the solution $\tpsi _i$ writes
\[
\tpsi _i (\omega) = \chi \left ( \Omega \cdot
\frac{\omega}{r}\right)c_i (\omega),\;\;c_i(\omega) = \frac{\omega
\cdot E_i}{\sqrt{r^2 - (\Omega \cdot \omega)^2}},\;\;\omega \in
\rsphere \setminus \{\pm r\Omega\}.
\]
\begin{lemma}
\label{InvField}
We consider the vector field $F$ given by
\[
F(\omega) = \sum _{i = 1} ^{d-1} \tpsi _i (\omega) E_i,\;\;\omega \in \rsphere.
\]
Then the vector field $F$ does not depend on the orthonormal basis
$\{E_1, ..., E_{d-1}\}$ of $(\R\Omega)^\bot$ and for any
orthogonal transformation $\calo$ of $\R^d$, preserving $\Omega$,
we have
\[
F(\calo \omega) = \calo F(\omega),\;\;\omega \in \rsphere.
\]
There is a function $\chi$ such that
\[
F(\omega)  = \chi \left ( \Omega \cdot \frac{\omega}{r}\right ) \frac{(I_d - \Omega \otimes \Omega)(\omega)}{\sqrt{r^2 - (\Omega \cdot \omega)^2}},\;\;\omega \in \rsphere \setminus \{\pm r \Omega\}
\]
and thus, for any $i \in \{1, ..., d-1\}$, we have
\[
\tpsi _i (\omega) = \chi \left ( \Omega \cdot \frac{\omega}{r}\right )\frac{\omega \cdot E_i}{\sqrt{r^2 - (\Omega \cdot \omega)^2}},\;\;\omega \in \rsphere \setminus \{\pm r \Omega\}.
\]
\end{lemma}
\begin{proof}
Consider any other orthonormal basis $\{F_1, ..., F_{d-1}\}$ of $(\R \Omega)^\bot$. Thanks to the identities
\[
E_1 \otimes E_1 + ...+E_{d-1} \otimes E_{d-1} + \Omega \otimes \Omega = I_d,\;\;F_1 \otimes F_1 + ...+F_{d-1} \otimes F_{d-1} + \Omega \otimes \Omega = I_d
\]
we obtain
\begin{align*}
\sum _{i = 1} ^{d-1} \tpsi _i E_i & = \sum _{i = 1} ^{d-1} \tpsi _{E_i} E_i = \sum _{i = 1} ^{d-1} \tpsi _{\sum _{j = 1} ^{d-1} (E_i \cdot F_j)F_j} E_i = \sum _{i = 1} ^{d-1} \sum _{j = 1} ^{d-1} (E_i \cdot F_j) \tpsi _{F_j} E_i \\
& = \sum _{j = 1} ^{d-1} \tpsi _{F_j} \sum _{i = 1} ^{d-1} (E_i \cdot F_j) E_i  = \sum _{j = 1} ^{d-1} \tpsi _{F_j} F_j .
\end{align*}
Pick $\calo$ any orthogonal transformation of $\R^d$, leaving invariant $\Omega$. For any $\omega \in \rsphere$, we can write, by Proposition \ref{Struct1}
\[
F(\calo \omega)  = \sum _{i = 1} ^{d-1} \tpsi _{E_i} (\calo \omega) E_i  = \sum _{i = 1} ^{d-1}\tpsi _{^t \calo E_i} (\omega) E_i = \calo \sum _{i = 1} ^{d-1}\tpsi _{^t \calo E_i } (\omega) \;^t \calo E_i = \calo F(\omega)
\]
where, in the last equality, we have used the independence of $F$
with respect to the orthonormal basis of $(\R\Omega)^\bot$. Take
now $\omega \in \rsphere \setminus \{\pm r \Omega\}$ and
\[
E = \frac{(I_d- \Omega \otimes \Omega)\omega}{\sqrt{r^2 - (\Omega
\cdot \omega)^2}}.
\]
Clearly $E \cdot \Omega = 0, |E| = 1$. \\
If $d = 2$, as we know that $F(\omega) \cdot \Omega = 0$, there is $\Lambda = \Lambda (\omega)$ such that
\[
F(\omega) = \Lambda (\omega) E = \Lambda (\omega) \frac{(I_2- \Omega \otimes \Omega)\omega}{\sqrt{r^2 - (\Omega \cdot \omega)^2}}.
\]
If $d \geq 3$, take any unitary vector $^\bot E$, orthogonal to
$E$ and $\Omega$, and consider the symmetry $ \calo  = I_d - 2
\;^\bot E \otimes ^\bot E. $ The above orthogonal transformation
leaves invariant $\Omega$, and thus, by the  hypothesis, we know
that $ F(\calo \omega ^\prime) = \calo F(\omega
^\prime),\;\;\omega ^\prime \in \rsphere. $ Observe that
\[
0 = {^\bot E} \cdot E = {^\bot E} \cdot \frac{\omega - (\omega
\cdot \Omega) \Omega}{\sqrt{r^2 - (\Omega \cdot \omega)^2}} =
\frac{^\bot E \cdot \omega}{\sqrt{r^2 - (\Omega \cdot
\omega)^2}},\;\mbox{implying that}\;\calo \omega = \omega\,,
\]
and thus
\[
F(\omega) = F(\calo \omega) = (I_d - 2 \;^\bot E \otimes \;^\bot E) F(\omega) = F(\omega) - 2(F(\omega) \cdot \;^\bot E) \;^\bot E.
\]
We deduce that $F(\omega) \cdot ^\bot\! E = 0$ for any vector
$^\bot E$, orthogonal to $E$ and $\Omega$. As $F(\omega) \cdot
\Omega = 0$, we deduce that $F(\omega)$ is orthogonal to any
vector orthogonal to $E$, anf thus there is $\Lambda = \Lambda
(\omega)$ such that
\[
F(\omega) = \Lambda (\omega) E  = \Lambda (\omega) \frac{(I_d - \Omega \otimes \Omega)\omega }{\sqrt{r^2 - (\Omega \cdot \omega)^2}},\;\;\omega \in \rsphere\setminus \{\pm r\Omega\}.
\]
We claim that $\Lambda (\omega)$ depends only on $\Omega \cdot
\frac{\omega}{r}$. Indeed, for any $d\geq 2$, and any orthogonal
transformation $\calo$, such that $\calo \Omega = \Omega$, we have
$F(\calo \omega) = \calo F(\omega)$,
\[
(I_d- \Omega \otimes \Omega) \calo \omega = \calo \omega - (\Omega
\cdot \calo \omega) \Omega = \calo \omega - (\Omega \cdot \omega)
\calo \Omega = \calo (I_d - \Omega \otimes \Omega) \omega,
\]
for all $\omega \in \rsphere\setminus \{\pm r\Omega\}$, and
\begin{align*}
\sqrt{r^2 - (\Omega \cdot \calo \omega)^2} & = |(I_d- \Omega \otimes \Omega) \calo \omega| = |\calo (I_d - \Omega \otimes \Omega) \omega|= |(I_d - \Omega \otimes \Omega) \omega|\\
& = \sqrt{r^2 - (\Omega \cdot  \omega)^2}\,,
\end{align*}
implying that $\Lambda (\calo \omega) = \Lambda (\omega), \omega \in \rsphere\setminus \{\pm r\Omega\}$. Actually, the previous equality holds true for any $\omega \in \rsphere$, since $\calo \Omega = \Omega$. We are done if we prove that $\Lambda (\omega) = \Lambda (\omega ^\prime)$ for any $\omega, \omegap \in \rsphere\setminus \{\pm r\Omega\}$ such that $\Omega \cdot \omega = \Omega \cdot \omegap, \omega \neq \omegap$. Consider the rotation $\calo $ such that
\[
\calo E = E^\prime,\;\;(\calo - I_d)|_{\mathrm{span}\{E, E^\prime \}^\bot} = 0,\;\;E = \frac{(I_d- \Omega \otimes \Omega)\omega}{\sqrt{r^2 - (\Omega \cdot \omega)^2}}, \;\;E^\prime = \frac{(I_d- \Omega \otimes \Omega)\omegap}{\sqrt{r^2 - (\Omega \cdot \omegap)^2}}.
\]
Notice that the condition $\calo E = E^\prime$ exactly says that $\calo \omega = \omegap$ and thus $\Lambda (\omegap) = \Lambda (\calo \omega) = \Lambda (\omega)$. We deduce that there is a function $\chi$ such that $\Lambda (\omega) = \chi \left ( \Omega \cdot \frac{\omega}{r}\right)$ and therefore
\[
\sum _{i = 1} ^{d-1} \tpsi _i (\omega) E_i = F(\omega) = \chi \left ( \Omega \cdot \frac{\omega}{r}\right)\frac{(I_d- \Omega \otimes \Omega)\omega}{\sqrt{r^2 - (\Omega \cdot \omega)^2}}= \sum _{i = 1} ^{d-1}\chi \left ( \Omega \cdot \frac{\omega}{r}\right)\frac{\omega \cdot E_i}{\sqrt{r^2 - (\Omega \cdot \omega)^2}}E_i
\]
implying that
\[
\tpsi _i (\omega) = \chi \left ( \Omega \cdot \frac{\omega}{r}\right)\frac{\omega \cdot E_i}{\sqrt{r^2 - (\Omega \cdot \omega)^2}},\;\;i \in \{1, ..., d-1\},\;\;\omega \in \rsphere\setminus \{\pm r\Omega\}.
\]
\end{proof}
\begin{remark}
\label{BiDim} In the case $d=2$, we take $E_1 = {^\bot \Omega},
\omega = r(\cos \theta \,\Omega + \sin \theta \;^\bot \Omega)$ and
therefore $\tpsi _1$ writes
\[
\tpsi _1 (r(\cos \theta \,\Omega + \sin \theta \;^\bot \Omega))=
\chi ( \cos \theta) \mathrm{sign} (\sin \theta),\;\;\theta \in
]-\pi, 0[ \;\cup \;]0,\pi[.
\]
Clearly, the function $\overline{\psi }_1 (\theta) := \tpsi _1
(r(\cos \theta \,\Omega + \sin \theta \;^\bot \Omega))$ is odd (in
particular $\int_{r \mathbb{S} ^1}\tpsi _1 (\omega) \dom =
\int_{-\pi}^\pi \overline{\psi }_1 (\theta) r\mathrm{d}\theta =
0$) and the condition
\[
\int_{r \mathbb{S} ^1}|\nabla _\omega \tpsi _1|^2 \mlo (\omega) \dom <+\infty
\]
implies that $\int _{-\pi} ^\pi |\partial _\theta \overline{\psi}_1 |^2 \mathrm{d}\theta <+\infty$. Therefore $\overline{\psi}_1$ is continuous on $]-\pi, \pi[$, and thus $\chi (1) = 0$. Notice that $\chi (-1) = 0$ as well, since $\lim_{\theta \nearrow \pi} \overline{\psi}_1 (\theta) = \tpsi _1 (- r \Omega) = \lim _{ \theta \searrow - \pi} \overline{\psi}_1 (\theta)$.
\end{remark}
Thanks to Lemma \ref{InvField}, in order to determine $\tpsi _i, i \in \{1, ..., d-1\}$, we only need to solve for $\chi$. The idea is to analyse the behavior of the functionals $J_{E_i}$ on the set of functions $\Psi _{i,h} (\omega) = h\left ( \Omega \cdot \frac{\omega}{r}\right)c_i (\omega),\omega \in \rsphere$. The notation $P_\omega$ stands for the orthogonal projection on the tangent space to $\rsphere$ at $\omega$, that is, $P_\omega = I_d - \frac{\omega \otimes \omega}{r^2}$.

\begin{pro}
The function $\chi$ constructed in Lemma \ref{InvField} solves the problem
\begin{equation}
\label{Prob2D} -\frac{\sigma}{r^2} \frac{\mathrm{d}}{\mathrm{d}c}
\left \{e^{lc} \chi ^{\;\prime} (c) (1-c^2)^{\frac{1}{2}} \right
\} = re^{lc},\;\;\chi (-1) = \chi (1) = 0
\end{equation}
for all $c \in ]-1,1[$, if $d=2$, and
\begin{equation}
\label{Prob3D} -\frac{\sigma}{r^2} \frac{\mathrm{d}}{\mathrm{d}c}
\left \{e^{lc} \chi ^{\prime}  (1-c^2)^{\frac{d-1}{2}}   \right \}
+ (d-2) \frac{\sigma}{r^2}e^{lc} \chi (c) (1-c^2)^{\frac{d-5}{2}}=
re^{lc}(1-c^2)^{\frac{d-2}{2}},
\end{equation}
for all $c \in ]-1,1[$, if $d\geq 3$.
\end{pro}
\begin{proof}
For any $i \in \{1,...,d-1\}$, the gradient of $\Psi _{i,h}$ writes
\[
\nabla _\omega \Psi _{i,h} = h^\prime \left ( \Omega \cdot \frac{\omega}{r}\right)c_i (\omega) \frac{P_\omega \Omega}{r} + h\left ( \Omega \cdot \frac{\omega}{r}\right)\nabla _\omega c_i
\]
where
\begin{align*}
\nabla _\omega c_i & = \frac{P_\omega E_i}{\sqrt{r^2 - (\omega \cdot \Omega)^2}} + \frac{(\omega \cdot E_i) (\omega \cdot \Omega)}{(r^2 - (\omega \cdot \Omega)^2)^{3/2}}P_\omega \Omega.
\end{align*}
Therefore we obtain
\begin{align*}
\nabla _\omega \psi _{i,h} = &\, h^\prime \left ( \Omega \cdot \frac{\omega}{r}\right)\frac{\omega \cdot E_i}{\sqrt{r^2 - (\omega \cdot \Omega)^2}} \frac{P_\omega \Omega}{r} \\
& + \frac{h \left ( \Omega \cdot \frac{\omega}{r}\right)}{\sqrt{r^2 - (\omega \cdot \Omega)^2}}\left [  P_\omega E_i + \frac{(\omega \cdot E_i) (\omega \cdot \Omega)}{r^2 - (\omega \cdot \Omega)^2}P_\omega \Omega \right ].
\end{align*}
Notice that $P_\omega \Omega$ and $\nabla _\omega c_i$ are orthogonal, thanks to the equality $|P_\omega \Omega |^2 = 1 - \frac{(\omega\cdot \Omega)^2}{r^2}$. Indeed, we have
\[
P_\omega \Omega \cdot \left [P_\omega E_i + \frac{(\omega \cdot E_i) (\omega \cdot \Omega)}{r^2 - (\omega \cdot \Omega)^2}P_\omega \Omega  \right ]= - \frac{(\omega \cdot E_i) (\omega \cdot \Omega)}{r^2} + \frac{(\omega \cdot E_i) (\omega \cdot \Omega)}{r^2 - (\omega \cdot \Omega)^2}|P_\omega \Omega|^2 = 0.
\]
Observe also that
\begin{align*}
|\nabla _\omega c_i|^2 & = \frac{1}{r^2 - (\omega \cdot \Omega)^2} \left [ 1 - \frac{(\omega \cdot E_i)^2}{r^2 - (\omega \cdot \Omega)^2}\right]
\end{align*}
implying that
\begin{align*}
|\nabla _\omega \Psi _{i,h} |^2 & = \left (h^\prime \left ( \Omega \cdot \frac{\omega}{r}\right)c_i (\omega) \right )^2\frac{|P_\omega \Omega |^2}{r^2} + \left (h \left ( \Omega \cdot \frac{\omega}{r}\right)\right ) ^2 |\nabla _\omega c_i |^2\\
& = \frac{\left (h^\prime \left ( \Omega \cdot \frac{\omega}{r}\right) \right )^2(\omega \cdot E_i)^2}{r^4} + \frac{\left (h \left ( \Omega \cdot \frac{\omega}{r}\right)\right ) ^2}{r^2 - (\omega \cdot \Omega)^2}\left [ 1 - \frac{(\omega \cdot E_i)^2}{r^2 - (\omega \cdot \Omega)^2}\right].
\end{align*}
Performing orthogonal changes of coordinates, which preserve
$\Omega$, we deduce that the integrals $\into{|\nabla _\omega \Psi
_{i,h}|^2 \mlo}$ do not depend on $i \in \{1, ..., d-1\}$, and
thus
\begin{align}
\label{EquFunct1}
\into{|\nabla _\omega \Psi _{i,h}|^2 \mlo} = &\, \frac{1}{d-1} \into{\frac{\left (h^\prime \left ( \Omega \cdot \frac{\omega}{r}\right) \right )^2}{r^4}[r^2 - (\omega \cdot \Omega)^2]\mlo} \\
& + \frac{d-2}{d-1} \into{\frac{\left (h \left ( \Omega \cdot
\frac{\omega}{r}\right) \right )^2}{r^2 - (\omega \cdot \Omega)^2}
\mlo}.\nonumber
\end{align}
We also need to compute the linear part of the functional $J_{E_i}$
\begin{align}
\label{EquFunct2} \into{\!\!\!\!\mlo \;(\omega \cdot E_i) h \left
(\Omega \cdot \frac{\omega}{r} \right )c_i (\omega)} & =
\into{\!\!\!\!\!\mlo\frac{h \left (\Omega \cdot \frac{\omega}{r}
\right )}{d-1} \sqrt{r^2 - (\omega \cdot \Omega)^2}}.
\end{align}
The expression of $J_{E_i} ( \psi _{i,h})$ follows by \eqref{EquFunct1}, \eqref{EquFunct2}
\begin{align*}
J_{E_i} (\psi _{i,h}) = &\, \frac{\sigma}{2(d-1)} \into{\!\!\!\mlo \left (  h ^\prime \left (\Omega \cdot \frac{\omega}{r} \right ) \right)^2 \frac{r^2 - (\Omega \cdot \omega)^2}{r^4}}\\
& + \frac{\sigma}{2} \frac{d-2}{d-1} \into{\!\!\!\mlo \frac{\left (  h  \left (\Omega \cdot \frac{\omega}{r} \right ) \right)^2}{r^2 - (\omega \cdot \Omega) ^2}}\\
& - \frac{1}{d-1} \into{\mlo h  \left (\Omega \cdot \frac{\omega}{r} \right )\sqrt{r^2 - (\omega \cdot \Omega)^2}} \\
= &\,\frac{\sigma}{2(d-1)r^2} \frac{\intth{e^{l \cos \theta} (h^\prime (\cos \theta))^2 \sin ^d \theta}}{\intth{e^{l \cos \theta }\sin ^{d-2}\theta}} \\
& + \frac{\sigma}{2} \frac{d-2}{d-1}\frac{\intth{e^{l\cos \theta} \left ( \frac{h(\cos \theta)}{r \sin \theta}\right ) ^2\sin ^{d-2} \theta}}{\intth{e^{l \cos \theta }\sin ^{d-2}\theta}} \\
& - \frac{1}{d-1} \frac{\intth{e^{l \cos \theta} h(\cos \theta) r \sin \theta \sin ^{d-2} \theta}}{\intth{e^{l \cos \theta }\sin ^{d-2}\theta}} \\
= &\, \frac{J(h)}{(d-1) \pi \beta _0 (l)}
\end{align*}
where $\pi \beta _0 (l) = \intth{e^{l \cos \theta} \sin ^{d-2}\theta }$ and
\begin{align*}
J(h) = &\,\frac{\sigma}{2r^2} \int _{-1} ^1 e^{lc} (h^\prime (c))^2 ( 1 - c^2 ) ^{\frac{d-1}{2}}\;\mathrm{d}c   + \frac{\sigma}{2}\frac{d-2}{r^2} \int _{-1} ^1 e^{lc} (h(c))^2 ( 1 - c^2 ) ^{\frac{d-5}{2}}\;\mathrm{d}c \\
& - r \int _{-1} ^1 e ^{lc} h(c) ( 1 - c^2 )
^{\frac{d-2}{2}}\;\mathrm{d}c.
\end{align*}
We consider the Hilbert spaces
\[
H_2 = \{ h :]-1, 1[ \to \R,\;\;(1 - c^2)^{1/4} h^\prime \in L^2 (]-1,1[),\;\;h(-1) = h(1) = 0\}
\]
and
\[
H_d = \{h : ]-1,1[ \to \R,\;(1 - c^2) ^{\frac{d-1}{4}} h^\prime
\in L^2(]-1,1[),\;\;(1-c^2) ^{\frac{d-5}{4}} h \in L^2 (]-1,1[)\},
\]
for $d \geq 3$, endowed with the scalar products
\[
(g,h)_2 = \int _{-1} ^1 g^\prime (c) h^\prime (c) \sqrt{1-c^2} \;\mathrm{d}c,\;\;g, h \in H_2
\]
and
\[
(g,h)_d = \int _{-1} ^1 \!\!\!g^\prime (c) h^\prime (c) (1-c^2)^{\frac{d-1}{2}} \;\mathrm{d}c + \int _{-1} ^1 \!\!\!g (c) h(c) (1-c^2)^{\frac{d-5}{2}} \;\mathrm{d}c,\;\;g, h \in H_d,\;\; \mbox{if } d \geq 3.
\]
By Lemma \ref{InvField}, there is a function $\chi$ such that $\tpsi _i = \chi \left ( \Omega \cdot \frac{\omega}{r}\right)c_i (\omega), i \in \{1,...,d-1\}$. We know that $\tpsi _i, i \in \{1, ..., d-1\}$, minimize the functionals $J_{E_i}(z)$, with $z \in H^1(\rsphere)$, $\into{z(\omega)} = 0$. In particular, for any $h \in H_d, d \geq 2$, we have
\[
J_{E_i} (\Psi _{i,h}) \geq J_{E_i} (\tpsi _i ),\;\;\Psi _{i,h}(\omega) = h \left ( \Omega \cdot \frac{\omega}{r}\right )c_i (\omega)
\]
implying that $\chi$, which belongs to $H_d$, is the solution of the minimization problem
\[
J(h) \geq J(\chi),\;\;h \in H_d.
\]
Thanks to the Lax-Milgram lemma, we deduce that $\chi$ is the solution of the problem \eqref{Prob2D} if $d = 2$, and \eqref{Prob3D} if $d\geq 3$.
\end{proof}
Up to now, for a given equilibrium $F = \mlo \;\dom$, we have determined the functions $\psi$ such that
\[
\intv{\psi (v) \lime\frac{\ave{Q(F+\eps \Go)} - \ave{Q(F)}}{\eps}} = 0
\]
for any bounded measure $\Go$, supported in $\rsphere$. But we need to control the linearization of $\ave{Q}$ around the equilibrium $F$ in the direction $\Fo$, which is not necessarily supported in $\rsphere$. It happens that the constraint $\Divv\{\Fo \abv\} = Q(F)$, see \eqref{Equ9}, will guarantee that
\[
\intv{\psi (v) \lime\frac{\ave{Q(F+\eps \Fo)} - \ave{Q(F)}}{\eps}} = \intv{\psi (v) \ave{\Divv A_F (\Fo)} } = 0.
\]
These computations are a little bit tedious and can be found in
\ref{C}. 
\begin{pro}
\label{ZeroAvePart}
Let $F = \mlo \dom$ be a von Mises-Fisher distribution with $l>0$, and $\Fo$ be a bounded measure (not charging a small neighborhood of $0$, for simplifying), satisfying $\Divv\{\Fo \abv \} = Q(F)$. Then the linearized of $\ave{Q}$ around $F$ in the direction $\Fo$ verifies
\[
\intv{\tpsi (v)\ave{\Divv A_F (\Fo)}\!\!} = 0,\mbox{ for any
generalized collision invariant } \tpsi \mbox{ of } \ave{Q}.
\]
\end{pro}


\section{The limit model}
\label{LimMod}

We identify the model satisfied by the limit distribution $f =
\lime \fe$. We already know that $f$ is a von Mises-Fisher
distribution $f = \rho (t,x) M_{l \Omega (t,x)}(\omega) \dom$ with
$\rho \geq 0, \Omega \in \sphere, l\geq 0, \lambda (l) =
\frac{\sigma}{r^2}l$. If $\frac{\sigma}{r^2} \geq \frac{1}{d}$,
then $l = 0$ and $\mlo \dom$ reduces to the isotropic measure on
$\rsphere$, that is $f = \rho (t,x) \frac{\dom}{r^{d-1} \bar\omega
_d}$, with zero mean velocity $u[f] = \into{\omega \rho \mlo } =
0$. In this case, the continuity equation reduces to the trivial
limit model $\partial _t \rho = 0, t \in \R_+$. From now on, we
assume that $\frac{\sigma}{r^2} \in ]0,\frac{1}{d}[$, and we
consider $l>0$ the unique solution for $\lambda (l) =
\frac{\sigma}{r^2}l$ cf. Proposition \ref{IsotropicEquilibrium}.
We are ready to justify the main result in Theorem
\ref{MainResult1} and the derivation of the SOH model
\eqref{Equ73}-\eqref{Equ74}.

\begin{proof} (of Theorem \ref{MainResult1})\\
The continuity equation \eqref{Equ73} comes from the continuity equation of \eqref{Equ71}
\[
\partial _t \intv{f} + \Divx \intv{fv} = \lime\left \{\partial _t \intv{\fe} + \Divx \intv{\fe v}\right \} = 0
\]
and the formula for the mean velocity of a von Mises-Fisher equilibrium
\[
u[f] = \into{\omega \rho \mlo} = \rho \frac{l\sigma}{r} \Omega = \rho \lambda (l) r \Omega.
\]
Equivalently, \eqref{Equ73} is obtained by using the collision
invariant $\tpsi = 1$. The equation \eqref{Equ74} will follow, by
using the $(d-1)$ dimensional linear space of collision invariants
studied in Proposition \ref{CollInvBis}. Revisiting the expansion
\eqref{FormalExpansion}, we obtain
\begin{equation}
\label{Equ75}
\partial _t f + \Divx(fv) + \Divv\{\ftw \abv\} = \Divv(A_f (\fo))
\end{equation}
together with the constraints
\begin{equation}
\label{Equ76}
\Divv\{f\abv\} = 0
\end{equation}
\begin{equation}
\label{Equ77}
\Divv\{\fo \abv\} = Q(f).
\end{equation}
The first constraint \eqref{Equ76} says that, for any $(t,x) \in \R_+ \times \R^d$, $\supp f(t,x) \subset \A$. Averaging the second constraint \eqref{Equ77} leads to
\[
\ave{Q(f)} = \ave{\Divv\{\fo \abv \}} = 0
\]
and thus $f(t,x) = \rho (t,x) M_{l\Omega (t,x)}(\omega) \dom,\omega \in \rsphere$. Averaging \eqref{Equ75} allows us to get rid of $\ftw$
\begin{equation}
\label{Equ78} \partial _t \ave{f} + \Divx \ave{vf} = \ave{\Divv
A_f (\fo)}.
\end{equation}
In order to eliminate $\fo$ as well, we test \eqref{Equ78} against
the functions $\psi _i (v) = \tpsi _i \left ( r \vsv\right)$,
where $(\tpsi _i)_{1\leq i \leq d-1}$ are the collision invariants
constructed in Proposition \ref{CollInvBis}. Indeed, by
Proposition \ref{ZeroAvePart}, we know that for any $i \in \{1,
..., d-1\}$
\[
\int_{v \neq 0} \tpsi_i \left ( r \vsv \right ) \ave{\Divv A_f
(\fo)} \;\dv = \int_{v \neq 0} \tpsi_i \left ( r \vsv \right )
\Divv A_f (\fo) \;\dv = I[\tpsi _i] = 0
\]
and therefore
\begin{equation}
\label{Equ79} \into{\partial _t (\rho \mlo ) \tpsi _i } + \into{\Divx (\rho \mlo \omega)\tpsi _i (\omega) } = 0,\;\;i \in \{1, ..., d-1\}.
\end{equation}
Let $\{E_1, ..., E_{d-1}, \Omega\}$ be a orthonormal basis and $\tpsi_1, ..., \tpsi _{d-1}$ be the solutions of the problems \eqref{Equ65}. We recall that
\[
\sum _{i =1 }^{d-1} \tpsi _i E_i = F(\omega) = \chi \left ( \Omega \cdot \frac{\omega}{r}\right) \frac{(I_d - \Omega \otimes \Omega) \frac{\omega}{r}}{\sqrt{1 - \left ( \Omega \cdot \frac{\omega}{r} \right )^2 } }.
\]
The equation \eqref{Equ79}, written for $i \in \{1, ..., d-1\}$, says that
\[
(I_d - \Omega \otimes \Omega) \into{[\partial _t (\rho \mlo) + \Divx (\rho \mlo \omega)]\frac{\chi \left ( \frac{\omega}{r} \cdot \Omega \right)}{\sqrt{1 - \left ( \Omega \cdot \frac{\omega}{r} \right )^2 }}\frac{\omega}{r}} = 0.
\]
We need to compute the vectors
\[
U_1 = \into{\partial _t \rho \mlo (\omega) \frac{\chi \left ( \frac{\omega}{r} \cdot \Omega \right)}{\sqrt{1 - \left ( \Omega \cdot \frac{\omega}{r} \right )^2 }}\frac{\omega}{r}}
\]
\[
U_2 = \into{\rho \mlo (\omega) \;l \; \partial _t \Omega \cdot \frac{\omega}{r} \frac{\chi \left ( \frac{\omega}{r} \cdot \Omega \right)}{\sqrt{1 - \left ( \Omega \cdot \frac{\omega}{r} \right )^2 }}\frac{\omega}{r}}
\]
\[
U_3 = \into{\omega \cdot \nabla _x \rho \mlo (\omega) \frac{\chi \left ( \frac{\omega}{r} \cdot \Omega \right)}{\sqrt{1 - \left ( \Omega \cdot \frac{\omega}{r} \right )^2 }}\frac{\omega}{r}}
\]
\[
U_4 = \into{l\rho \omega \cdot \;^t \partial _x \Omega \frac{\omega}{r} \;\mlo (\omega) \frac{\chi \left ( \frac{\omega}{r} \cdot \Omega \right)}{\sqrt{1 - \left ( \Omega \cdot \frac{\omega}{r} \right )^2 }}\frac{\omega}{r}}
\]
and to impose
\begin{equation}
\label{Equ85} \sum _{i = 1} ^4 (I_d - \Omega \otimes \Omega)U_i = 0.
\end{equation}
Clearly, the first vector $U_1$ is parallel to $\Omega$, and thus
\begin{equation}
\label{Equ86}
(I_d - \Omega \otimes \Omega)U_1 = 0.
\end{equation}
The treatment of the second and third vectors requires to compute
\begin{align*}
{\mathcal A} := &\, \into{\frac{\omega}{r} \otimes \frac{\omega}{r} \mlo (\omega) \frac{\chi \left ( \frac{\omega}{r} \cdot \Omega \right)}{\sqrt{1 - \left ( \Omega \cdot \frac{\omega}{r} \right )^2 }}} \\
= &\,\sum _{i = 1}^{d-1} \into{\frac{(\omega \cdot E_i)^2}{r^2} \mlo (\omega) \frac{\chi \left ( \frac{\omega}{r} \cdot \Omega \right)}{\sqrt{1 - \left ( \Omega \cdot \frac{\omega}{r} \right )^2 }}} \;E_i \otimes E_i \\
& + \into{\frac{(\omega \cdot \Omega)^2}{r^2} \mlo (\omega) \frac{\chi \left ( \frac{\omega}{r} \cdot \Omega \right)}{\sqrt{1 - \left ( \Omega \cdot \frac{\omega}{r} \right )^2 }}}\;\Omega \otimes \Omega \\
= &\,\frac{1}{d-1} \into{\left [1 - \left ( \Omega \cdot \frac{\omega}{r}\right)^2   \right]\mlo (\omega) \frac{\chi \left ( \frac{\omega}{r} \cdot \Omega \right)}{\sqrt{1 - \left ( \Omega \cdot \frac{\omega}{r} \right )^2 }}}\sum_{i = 1}^{d-1}E_i \otimes E_i \\
& + \into{\frac{(\omega \cdot \Omega)^2}{r^2} \mlo (\omega)  \frac{\chi \left ( \frac{\omega}{r} \cdot \Omega \right)}{\sqrt{1 - \left ( \Omega \cdot \frac{\omega}{r} \right )^2 }}}\;\Omega \otimes \Omega \\
= &\,\frac{\intth{\sin ^2 \theta e^{l\cos \theta} \;\frac{\chi (\cos \theta)}{\sin \theta} \sin ^{d-2}\theta}}{\intth{e^{l\cos \theta}\sin ^{d-2} \theta}}\;\frac{I_d - \Omega \otimes \Omega}{d-1} \\
& + \frac{\intth{\cos ^2 \theta e^{l\cos \theta} \;\frac{\chi
(\cos \theta)}{\sin \theta} \sin ^{d-2}\theta}}{\intth{e^{l\cos
\theta}\sin ^{d-2} \theta}}\Omega \otimes \Omega.
\end{align*}
We obtain, thanks to the identity $\partial _t \Omega \cdot \Omega = \frac{1}{2}\partial _t |\Omega |^2 = 0$
\begin{equation}
\label{Equ87}
(I_d - \Omega \otimes \Omega) U_2 = (I_d - \Omega \otimes \Omega) \rho l {\mathcal A} \partial _t \Omega =\frac{ \rho l}{d-1}\frac{\intth{ e^{l\cos \theta} \chi (\cos \theta)\sin ^{d-1}\theta}}{\intth{e^{l\cos \theta}\sin ^{d-2} \theta}}\partial _t \Omega
\end{equation}
and
\begin{align}
\label{Equ88}
(I_d - \Omega \otimes \Omega) U_3  & = r (I_d - \Omega \otimes \Omega) {\mathcal A}\nabla _x \rho \nonumber \\
& = \frac{r}{d-1} \frac{\intth{e^{l\cos \theta} \chi (\cos \theta)\sin ^{d-1}\theta}}{\intth{e^{l\cos \theta}\sin ^{d-2} \theta}}(I_d - \Omega \otimes \Omega) \nabla _x \rho.
\end{align}
We concentrate now on the last vector $U_4$. Observe that
\[
(I_d - \Omega \otimes \Omega) U_4 = r\rho l \into{\frac{\omega}{r} \otimes \frac{\omega}{r} : \partial _x \Omega \;\mlo (\omega) \frac{\chi \left ( \frac{\omega}{r} \cdot \Omega \right)}{\sqrt{1 - \left ( \Omega \cdot \frac{\omega}{r} \right )^2 }}\sum _{i = 1} ^{d-1} \left ( E_i \cdot \frac{\omega}{r} \right)E_i}
\]
and for any $i \in \{1, ..., d-1\}$
\begin{align*}
& \into{\frac{\omega}{r} \otimes \frac{\omega}{r} : \partial _x \Omega \;\mlo (\omega)\frac{\chi \left ( \frac{\omega}{r} \cdot \Omega \right)}{\sqrt{1 - \left ( \Omega \cdot \frac{\omega}{r} \right )^2 }}\left (E_i \cdot \frac{\omega}{r}\right )}\\
& = \into{\frac{(\omega \cdot E_i)^2 \;( \omega \cdot \Omega)}{r^3} \mlo (\omega) \frac{\chi \left ( \frac{\omega}{r} \cdot \Omega \right)}{\sqrt{1 - \left ( \Omega \cdot \frac{\omega}{r} \right )^2 }}}\;[E_i \otimes \Omega : \partial _x \Omega + \Omega \otimes E_i : \partial _x \Omega]\\
& = \into{\!\!\!\!\!\!\!\!\frac{1 - \left ( \Omega \cdot \frac{\omega}{r}\right)^2}{d-1}\; \frac{(\omega \cdot \Omega)}{r} \mlo (\omega) \frac{\chi \left ( \frac{\omega}{r} \cdot \Omega \right)}{\sqrt{1 - \left ( \Omega \cdot \frac{\omega}{r} \right )^2 }}}\;[E_i \otimes \Omega : \partial _x \Omega + \Omega \otimes E_i : \partial _x \Omega]\\
& = \frac{1}{d-1} \frac{\intth{\sin ^2 \theta \cos \theta e ^{l \cos \theta}\; \frac{\chi (\cos \theta)}{\sin \theta} \sin ^{d-2} \theta}}{\intth{e^{l\cos \theta} \sin ^{d-2} \theta}}(\partial _x \Omega \Omega \cdot E_i + \;^t\partial _x \Omega \Omega \cdot E_i).
\end{align*}
Thanks to the formula $^t \partial _x \Omega \Omega = \frac{1}{2}\nabla _x |\Omega |^2 = 0$, we obtain
\begin{align}
\label{Equ89}
(I_d - \Omega \otimes \Omega) U_4 & = \frac{r \rho l}{d-1} \frac{\intth{\cos \theta e ^{l \cos \theta} \chi (\cos \theta) \sin ^{d-1} \theta}}{\intth{e^{l\cos \theta} \sin ^{d-2}
\theta}}\sum _{i = 1} ^{d-1} (\partial _x \Omega \Omega \cdot E_i)E_i\nonumber\\
& = \frac{r \rho l}{d-1} \frac{\intth{\cos \theta e ^{l \cos \theta} \chi (\cos \theta) \sin ^{d-1} \theta}}{\intth{e^{l\cos \theta} \sin ^{d-2} \theta}}
(I_d - \Omega \otimes \Omega)\partial _x \Omega \Omega \nonumber \\
& = \frac{r \rho l}{d-1} \frac{\intth{\cos \theta e ^{l \cos
\theta} \chi (\cos \theta) \sin ^{d-1} \theta}}{\intth{e^{l\cos
\theta} \sin ^{d-2} \theta}}\partial _x \Omega \Omega.
\end{align}
The evolution equation for the orientation $\Omega$ comes now by
collecting \eqref{Equ85}, \eqref{Equ86}, \eqref{Equ87},
\eqref{Equ88} and \eqref{Equ89} to get
\begin{align*}
& \frac{ \rho l \partial _t \Omega + r(I_d - \Omega \otimes \Omega)  \nabla _x \rho  }{d-1} \;\frac{\intth{e ^{l \cos \theta} \chi (\cos \theta) \sin ^{d-1} \theta}}{\intth{e^{l\cos \theta} \sin ^{d-2} \theta}}\\
& + \frac{r\rho l}{d-1}  \frac{\intth{\cos \theta e ^{l \cos \theta} \chi (\cos \theta) \sin ^{d-1} \theta}}{\intth{e^{l\cos \theta} \sin ^{d-2} \theta}}\partial _x \Omega \Omega = 0
\end{align*}
which also rewrites as
\[
\partial _t\Omega +  r\frac{\intth{\cos \theta e ^{l \cos \theta} \chi (\cos \theta) \sin ^{d-1} \theta}}{\intth{e^{l\cos \theta} \chi (\cos \theta) \sin ^{d-1} \theta}}(\Omega \cdot \nabla _x)\Omega + \frac{r}{l}(I_d - \Omega \otimes \Omega) \frac{\nabla _x \rho }{\rho} = 0.
\]
\end{proof}
\begin{remark}
Taking the scalar product of the equation $\eqref{Equ74}$ with $\Omega$, we obtain
\[
\frac{1}{2}\partial _t |\Omega |^2 + \frac{k_d r}{2} (\Omega \cdot \nabla _x) |\Omega |^2 = 0,\;\;(t,x) \in \R_+ \times \R^d
\]
implying that $|\Omega (t,x)| = 1, (t,x) \in \R_+ \times \R^d$, provided that $|\Omega (0,x) | = 1, x \in \R^d$.
\end{remark}


\appendix

\section{Integration by parts on spheres}
\label{A}
\begin{proof} (of Lemma \ref{IntByPartsSphere})\\
We pick a function $\eta \in C^1 _c (]r_1, r_2[)$ and observe that
\[
\Divv \{ \eta (|v|) A(v) \} = \eta ^\prime (|v|) \vsv \cdot A(v) + \eta (|v|) (\Divv A)(v),\;\;v \in {\mathcal O}.
\]
Integrating with respect to $v$ over ${\mathcal O}$ leads to
\begin{align*}
0 & = \int _{{\mathcal O}} \Divv \{ \eta (|v|) A(v) \}\;\dv  = \int _{{\mathcal O}}\eta ^\prime (|v|) \vsv \cdot A(v)\;\dv + \int _{{\mathcal O}}\eta (|v|) (\Divv A)(v)\;\dv \\
& = \int _{r_1} ^{r_2} \eta ^\prime (t) \int _{|\omega | = 1} \omega \cdot A(t\omega) t ^{d-1} \;\dom \mathrm{d}t + \int _{r_1} ^{r_2} \eta (t) \int _{|\omega | = 1} (\Divv A ) (t \omega) t ^{d-1} \;\dom \mathrm{d}t \\
& = \int _{r_1} ^{r_2}\eta (t) \left [- \frac{\mathrm{d}}{\mathrm{d}t} \int _{|\omega | = 1}\omega \cdot A(t \omega) t ^{d-1}\;\dom + \int _{|\omega | = 1}(\Divv A )(t\omega)  t ^{d-1}\;\dom   \right ]\;\mathrm{d}t.
\end{align*}
We deduce that
\begin{align*}
\int _{|\omega |  = t}(\Divv A) (\omega) \;\dom & = \frac{\mathrm{d}}{\mathrm{d}t} \int _{|\omega | = 1}\omega \cdot A(t \omega) t ^{d-1}\;\dom  \\
& = \int _{|\omega |  = 1} \{ \omega \cdot \partial _v A (t\omega) \omega t ^{d-1} + \omega \cdot A(t\omega) (d-1) t ^{d-2} \} \;\dom \\
& = \int_{|\omega | = t}\left \{\frac{\omega \otimes \omega}{t^2} : \partial _v A (\omega) + \frac{(d-1)\omega}{t^2} \cdot A(\omega) \right \} \;\dom.
\end{align*}
Assume now that $A(v) \cdot v = 0, v \in {\mathcal O}$. Taking the gradient with respect to $v$ yields $^t \partial _v A(v) v + A(v) = 0$ implying $\partial _v A(v) : v \otimes v = - A(v) \cdot v = 0, v \in {\mathcal O}$. In this case \eqref{Ident0} reduces to \eqref{Ident1}. The formula in \eqref{Ident2} follows easily by applying \eqref{Ident1} with the field $v \to \chi (v) A(v)$.
\end{proof}

\section{Differential operators on spheres}
\label{B}
\begin{proof} (of Lemma \ref{Extension})\\
1. Pick a point $\omega \in \rsphere$ and a tangent vector $X \in
T_\omega (\rsphere)$. Let $\gamma : ]-\eps, \eps[ \to \rsphere$ be
a smooth curve such that $\gamma (0) = \omega$, $\gamma
^{\;\prime} (0) = X$. Then we have
\begin{align*}
\nabla _\omega \tpsi \cdot X & = \mathrm{d} \tpsi _\omega (X) = \frac{\mathrm{d}}{\mathrm{d}t}_{|_{t = 0}} \tpsi (\gamma (t)) = \frac{\mathrm{d}}{\mathrm{d}t}_{|_{t = 0}} \psi (\gamma (t)) \\
& = \widetilde{\nabla _v \psi} (\omega) \cdot X = \imoo \widetilde{\nabla _v \psi} (\omega) \cdot X
\end{align*}
saying that
\[
\nabla _\omega \tpsi  - \imoo \widetilde{\nabla _v \psi } \in T_\omega (\rsphere) \cap (T_\omega (\rsphere ))^\bot = \{0\}.
\]
Therefore we deduce that $\nabla _\omega \tpsi = \imoo \widetilde{\nabla _v \psi}$. \\
2. For any $\omega _t \in t \sphere$ and $X \in T_{\omega _t} (t
\sphere)$, pick a smooth curve $\gamma : ]-\eps, \eps[ \to t
\sphere$ such that $\gamma (0) = \omega _t, \gamma ^{\;\prime} (0)
= X$. Therefore we have
\begin{align*}
\nabla _{\omega _t} \widetilde\psi^t (\omega _t) \cdot X & =
\frac{\mathrm{d}}{\mathrm{d}s }_{| _{s = 0}} \psi (\gamma (s)) =
\frac{\mathrm{d}}{\mathrm{d}s }_{| _{s = 0}} \tpsi \left (r
\frac{\gamma (s)}{t}\right ) = \nabla _\omega \tpsi \left ( r
\frac{\omega _t}{t} \right ) \cdot \frac{r}{t} X
\end{align*}
saying that $(\nabla _{\omega _t} \widetilde\psi^t ) (\omega _t) =
\frac{r}{t} ( \nabla _\omega \tpsi) \left (r \frac{\omega _t}{t}
\right )$. Actually the function $\psi$ has only tangent gradient
(to the spheres), and thus
\[
(\nabla _v \psi )(\omega _t) = (\nabla _{\omega _t}
\widetilde\psi^t) (\omega _t) = \frac{r}{t} (\nabla _\omega \tpsi
) \left ( r \frac{\omega _t}{t} \right),\;\;|\omega _t | = t.
\]
3. Consider a $C^1$ function $\tpsi$ on $\rsphere$ and $\psi$ a $C^1$ extension of $\tpsi$ on $\calo$. By Lemma \ref{IntByPartsSphere}, we know that
\begin{equation}
\label{EquStep1}
\int_{|\omega| = r} \widetilde{\nabla _v \psi }(\omega) \cdot \txi (\omega) \;\dom + \int_{|\omega| = r} \tpsi (\omega) \widetilde{\Divv \xi } (\omega) \;\dom = 0.
\end{equation}
But, by the previous statement, we can write
\begin{align}
\label{EquStep2}
\widetilde{\nabla _v \psi } (\omega) \cdot \txi (\omega) & = \widetilde{\nabla _v \psi }(\omega) \cdot \imoo \txi (\omega) = \imoo \widetilde{\nabla _v \psi }(\omega)  \cdot \tilde{\xi}(\omega) \nonumber \\
& = \nabla _\omega \tpsi (\omega) \cdot \txi (\omega).
\end{align}
Combining \eqref{EquStep1}, \eqref{EquStep2} yields
\begin{align*}
\int_{|\omega | = r} \tpsi (\omega) \Divo \txi (\omega) \;\dom & = - \int_{|\omega | = r} \nabla _\omega \tpsi (\omega) \cdot \txi (\omega) \;\dom\\
& = \int_{|\omega | = r}\tpsi (\omega)  \widetilde{\Divv \xi }(\omega) \;\dom,\;\;\tpsi \in C^1 (\rsphere)
\end{align*}
implying that $\Divo \txi = \widetilde{\Divv \xi }$. \\
4. Consider $\txi = \txi (\omega)$ a $C^1$ tangent vector field on
$\rsphere$ and $\xi (v) = \txi\left ( r \vsv \right ), v \in
\R^d\setminus \{0\}$. We have $\xi (v) \cdot v = 0, v \in \R ^d
\setminus \{0\}$, and for any $t>0$
\[
(\Divv \xi ) (\omega _t) = (\mathrm{div} _{\omega _t}
\widetilde\xi^t )(\omega _t) = \frac{r}{t} ( \Divo \txi ) \left (
r \frac{\omega _t }{t} \right ),\;\;\omega _t \in t \sphere.
\]
The first equality comes by the third statement of Lemma
\ref{Extension}. In oder to check the second equality, pick a
$C^1$ function $\widetilde\psi^t$ on $t \sphere$ and consider the
function $\tpsi (\omega) = \widetilde\psi^t (t\omega /r), \omega
\in \rsphere$. We have
\[
\nabla _\omega \tpsi (\omega) = \frac{t}{r}( \nabla _{\omega _t}
\widetilde\psi^t ) \left ( t \frac{\omega}{r} \right )
\]
and thus
\begin{align*}
- \int _{|\omega _t| = t}  ( \mathrm{div} _{\omega _t}
\widetilde\xi^t ) (\omega _t) \widetilde\psi^t (\omega _t)
\;\mathrm{d}\omega _t &
= \int _{|\omega _t| = t}\widetilde\xi^t (\omega _t) \cdot \nabla _{\omega _t} \widetilde\psi^t (\omega _t)\;\mathrm{d}\omega _t \\
& = \int_{|\omega | = r} \xi \left ( t \frac{\omega}{r}\right ) \cdot \left ( \nabla _{\omega _t} \widetilde\psi^t \right ) \left ( t \frac{\omega}{r} \right ) \left  ( \frac{t}{r} \right ) ^{d-1} \;\dom\\
& =  \int_{|\omega | = r} \txi (\omega) \cdot \nabla _\omega \tpsi (\omega) \left ( \frac{t}{r}\right ) ^{d-2} \;\dom\\
& = - \int_{|\omega | = r} (\Divo \txi ) (\omega) \tpsi (\omega) \left ( \frac{t}{r}\right ) ^{d-2}\;\dom \\
& = - \int_{|\omega _t| = t}\frac{r}{t} (\Divo \txi ) \left ( r
\frac{\omega _t}{t} \right )\widetilde\psi^t (\omega _t)
\;\mathrm{d}\omega _t.
\end{align*}
We deduce that $(\mathrm{div}_{\omega _t } \xi )(\omega _t) =
\frac{r}{t} (\Divo \txi ) (r \omega _t /t)$ for any $\omega _t \in
t \sphere$.
\end{proof}


\section{Collision invariants and linearization of $\ave{Q}$}
\label{C}
\begin{proof} (of Proposition \ref{ZeroAvePart})\\
Consider a collision invariant $\tpsi$, and let us compute
\[
I[\tpsi] := \int _{v \neq 0} \tpsi \left ( r \vsv \right )\Divv A_F (\Fo) \;\dv
\]
that is
\begin{align*}
I[\tpsi] =& \, \int _{v \neq 0} \left \{ - (v - u[F] ) \cdot \nabla _v \left[ \tpsi \left ( r \vsv \right )\right] + \sigma \Delta _v \left [ \tpsi \left ( r \vsv \right )\right ]\right \} \Fo \;\dv \\
& + \int _{v \neq 0} (v - u[F]) \cdot \frac{\int _{v^\prime \neq
0}\nabla _{v^\prime} \left[ \tpsi \left ( r
\frac{v^\prime}{|v^\prime|} \right )\right]F
\;\mathrm{d}v^\prime}{\intvp{F}}\Fo\;\dv.
\end{align*}
We consider the application
\begin{align*}
\label{EquChi}
\chi (v) = &\,- (v-u[F])\cdot \nabla _v \left[ \tpsi \left ( r \vsv \right )\right] + \sigma \Delta _v \left[ \tpsi \left ( r \vsv \right )\right] \\
& + (v-u[F]) \cdot \frac{\int _{v^\prime \neq 0}\nabla _{v^\prime} \left[ \tpsi \left ( r \frac{v^\prime}{|v^\prime|} \right )\right]F \;\mathrm{d}v^\prime}{\intvp{F}}\nonumber \\
= &\, u[F] \cdot \nabla _v \left[ \tpsi \left ( r \vsv \right )\right] + \sigma \Delta _v \left[ \tpsi \left ( r \vsv \right )\right]\nonumber \\
& + (v-u[F]) \cdot \intopr{(\nabla _{\omega ^\prime} \tpsi ) (\omega ^\prime) \mlo (\omega ^{\prime})},\;\;v \neq 0.\nonumber
\end{align*}
As $\tpsi$ is a collision invariant, we have $\chi (\omega) = 0$, for any $\omega \in \rsphere$ cf. \eqref{Equ43}. Thanks to Lemma \ref{FOneF}, the integral $I[\tpsi]$ can be written
\begin{align*}
I[\tpsi] & = \int_{v \neq 0} \chi (v) \Fo \;\dv = \frac{\sigma}{\beta} \frac{\mlo}{M} \frac{\mathrm{d}}{\mathrm{d}t}_{|_{t = r}} \int _{|\omega _t| = t} M (\omega _t) \frac{\chi (\omega _t)}{t(t^2 - r^2)}\;\mathrm{d}\omega _t\\
& = \frac{\sigma}{\beta} \frac{\mlo}{M}
\frac{\mathrm{d}}{\mathrm{d}t}_{|_{t = r}} \int _{|\omega | = r}
M\left ( t \frac{\omega}{r}\right ) \frac{\chi \left (
t\frac{\omega}{r}\right)}{t(t^2 - r^2)} \left ( \frac{t}{r} \right
) ^{d-1} \;\mathrm{d}\omega.
\end{align*}
Thanks to the second statement in Lemma \ref{Extension}, we can write
\[
\nabla _v \left [ \tpsi \left ( r \vsv \right ) \right ]\left ( t \frac{\omega}{r}\right ) = \frac{r}{t} ( \nabla _\omega \tpsi ) (\omega)
\]
and by \eqref{Complement} in Lemma \ref{Extension} point 4, we
have
\[
\Delta _v \left [ \tpsi \left ( r \vsv \right ) \right ]\left ( t \frac{\omega}{r}\right ) = \left ( \frac{r}{t} \right )^2( \Delta _\omega \tpsi ) (\omega).
\]
Therefore, the function $t \to \chi \left ( t \frac{\omega}{r}\right)$ is given by
\[
\chi \left ( t \frac{\omega}{r}\right) = \frac{r}{t} u[F] \cdot ( \nabla _\omega \tpsi ) (\omega) + \sigma \frac{r^2}{t^2} ( \Delta _\omega \tpsi ) (\omega) + \left ( t \frac{\omega}{r} - u[F] \right ) \cdot W[\tpsi]
\]
with $W[\tpsi] = \into{\nabla _\omega \tpsi \mlo (\omega)}$. As $\chi (\omega) = 0, \omega \in \rsphere$, because $\tpsi$ is a collision invariant, we obtain
\begin{align*}
M\left ( t \frac{\omega}{r} \right ) & \frac{\chi \left ( t \frac{\omega}{r} \right )}{t(t^2 - r^2)}  = M\left ( t \frac{\omega}{r} \right ) \frac{\chi \left ( t \frac{\omega}{r} \right )- \chi (\omega)}{t(t^2 - r^2)}\\
& = M\left ( t \frac{\omega}{r} \right ) \frac{\frac{r-t}{t} u[F] \cdot (\nabla _\omega \tpsi )(\omega) + \sigma \frac{r^2 - t^2}{t^2} (\Delta _\omega \tpsi ) (\omega) + \frac{t-r}{r} \omega \cdot W[\tpsi ]}{t(t-r)(t+r)}\\
& = M\left ( t \frac{\omega}{r} \right ) \frac{\omega \cdot W[\tpsi]}{rt(t+r)} - M\left ( t \frac{\omega}{r} \right ) \frac{\sigma}{t^3}(\Delta _\omega \tpsi )(\omega) - M\left ( t \frac{\omega}{r} \right ) \frac{u[F] \cdot (\nabla _\omega \tpsi)(\omega)}{t^2 (t+r)}\\
& = \frac{M\left ( t \frac{\omega}{r} \right ) }{rt(t+r)} [ \omega \cdot W[\tpsi] + u[F] \cdot (\nabla _\omega \tpsi )(\omega)] - \frac{\sigma}{t^3}\Divo \left ( M\left ( t \frac{\omega}{r} \right ) \nabla _\omega \tpsi \right ).
\end{align*}
It is easily seen that $\into{M\left ( t \frac{\omega}{r} \right ) \omega } \in \R \Omega$ and, as we know that $W[\tpsi]\in (\R \Omega)^\bot$, we deduce that
\[
\into{M\left ( t \frac{\omega}{r} \right ) \omega \cdot W[\tpsi]} = 0.
\]
Taking into account that
\[
\into{\Divo\left \{ M\left ( t \frac{\omega}{r} \right ) \nabla _\omega \tpsi \right\}} = 0
\]
we deduce that
\begin{align*}
I[\tpsi] = &\, \frac{\sigma}{\beta} \frac{\mlo}{M} \frac{\mathrm{d}}{\mathrm{d}t}_{|_{t=r}} \left [ \left ( \frac{t}{r} \right ) ^{d-1}  \into{\frac{M\left ( t \frac{\omega}{r} \right ) \nabla _\omega \tpsi \cdot u[F]}{rt(t+r)}}\right]\\
= &\, \frac{\sigma}{\beta} \frac{\mlo}{M}\frac{\mathrm{d}}{\mathrm{d}t}_{|_{t=r}} \left [ \left ( \frac{t}{r} \right ) ^{d-1} \frac{1}{rt(t+r)}\right] \into{M\left (  \omega \right ) \nabla _\omega \tpsi \cdot u[F]} \\
& + \frac{\sigma}{2 r^3 \beta}
\frac{\mlo}{M}\frac{\mathrm{d}}{\mathrm{d}t}_{|_{t=r}}
\into{M\left ( t \frac{\omega}{r} \right ) \nabla _\omega \tpsi
\cdot u[F]}.
\end{align*}
As before
\[
\frac{\mlo}{M}\into{M(\omega) \nabla _\omega \tpsi} \cdot u[F] = \into{\mlo \nabla _\omega \tpsi } \cdot u[F] = W[\tpsi] \cdot u[F] = 0
\]
implying that
\begin{align*}
I[\tpsi] & = \frac{\sigma}{2r^3 \beta} \frac{\mlo}{M} \frac{\mathrm{d}}{\mathrm{d}t}_{|_{t=r}} \into{M\left ( t \frac{\omega}{r} \right ) \nabla _\omega \tpsi \cdot u[F]}\\
& = \frac{\sigma}{2r^3 \beta} \frac{\mlo}{M}\into{M(\omega) \left ( \frac{u[F] - \omega}{\sigma} \cdot \frac{\omega}{r} \right ) \;\left ( \nabla _\omega \tpsi \cdot u[F]\right )}\\
& = \frac{1}{2r^4 \beta}\into{\mlo (u[F] \cdot \omega  - r^2) ( \nabla _\omega \tpsi \cdot u[F])}\\
& = \frac{1}{2r^4 \beta}\into{\mlo (\nabla _\omega \tpsi \cdot
u[F] ) (\omega \cdot u[F])}.
\end{align*}
In the last equality we have used one more time that $W[\tpsi] \cdot u[F] = 0$. We claim that the last integral vanishes. Indeed, multiplying by $(\omega \cdot u[F])^2$ the equation \eqref{Equ62} satisfied by the collision invariant $\tpsi$ one gets
\begin{align*}
2\sigma \into{\mlo (\nabla _\omega \tpsi \cdot u[F] ) (\omega \cdot u[F])} & = W[\tpsi] \cdot \into{\mlo (\omega \cdot u[F])^2 (\omega - u[F])} \\
& = W[\tpsi] \cdot \into{\mlo (\omega \cdot u[F])^2 \omega }.
\end{align*}
It is easily seen that $\into{\mlo (\omega \cdot u[F])^2\omega } \in \R \Omega$ and therefore
\[
W[\tpsi] \cdot \into{\mlo (\omega \cdot u[F])^2 \omega } = 0
\]
saying that $I[\tpsi] = 0$.
\end{proof}

\section*{Acknowledgments}
MB acknowleges support from the Euratom research and training programme 2014-2018 under grant agreement No 633053. JAC acknowleges partial support of the Royal Society via a Wolfson Research Merit Award.

\end{document}